\documentclass[11pt,english]{amsart}
\usepackage[T1]{fontenc}
\usepackage[latin9]{inputenc}
\usepackage{geometry}
\usepackage[active]{srcltx}
\usepackage{float}
\usepackage{amstext}
\usepackage{amssymb}
\usepackage{graphicx}
\usepackage{comment}
\usepackage{color,soul}
\usepackage{hyperref}

\makeatletter

\usepackage{amsthm}\usepackage{amstext}

\usepackage{enumitem}
\theoremstyle{plain}
\newtheorem{thm}{\protect\theoremname}
 \theoremstyle{definition}
 \newtheorem*{defn*}{\protect\definitionname}
  \theoremstyle{remark}
  \newtheorem*{rem*}{\protect\remarkname}
  \theoremstyle{plain}
  \newtheorem{cor}[thm]{\protect\corollaryname}
  \theoremstyle{plain}
  \newtheorem{lem}[thm]{\protect\lemmaname}
  \theoremstyle{plain}
  \newtheorem{prop}[thm]{\protect\propositionname}
  \theoremstyle{definition}
  \newtheorem{example}[thm]{\protect\examplename}

\numberwithin{thm}{section}
\numberwithin{equation}{section}

\newcommand{\Des}{{\rm Des}\,}
\newcommand{\des}{{\rm des}\,}
\newcommand{\la}{\lambda}

\newcommand{\asc}{{\rm asc}\,}
\newcommand{\s}{\mathbf{s}}

\newcommand{\Z}{\mathbb{Z}}

\newcommand{\Asc}{{\rm Asc}\,}

\newcommand{\e}{\mathbf{e}}
\newcommand{\bv}{\mathbf{v}}
\newcommand{\Co}{\mathbf{C}}

\newcommand{\Qn}{Q_{n}^{(\mathbf{s})}}
\newcommand{\Qnn}{Q_{n-1}^{(\mathbf{s})}}
\newcommand{\In}{I_{n}^{(\mathbf{s})}}
\newcommand{\Inn}{I_{n-1}^{(\mathbf{s})}}

\newcommand{\ax}{\left[s_{n}\right]_{x}}

\newcommand{\anx}{\left[n\right]_{x}}

\usepackage{babel}
\providecommand{\corollaryname}{Corollary}
  \providecommand{\definitionname}{Definition}
  \providecommand{\examplename}{Example}
  \providecommand{\lemmaname}{Lemma}
  \providecommand{\propositionname}{Proposition}
  \providecommand{\remarkname}{Remark}
\providecommand{\theoremname}{Theorem}

\usepackage{babel}

\makeatother

\usepackage{babel}

\begin{document}

\title{Coefficients of the inflated Eulerian polynomial}

\author{Juan S. Auli}
\address{Department of Mathematics\\
         Dartmouth College\\
         Hanover, NH 03755\\
         U.S.A.}
\email{juan.s.auli.gr@dartmouth.edu}

\author{Ron Graham}
\address{Department of Computer Science and Engineering\\
UC San Diego\\
La Jolla, CA 92093-0404\\
U.S.A.}
\email{graham@ucsd.edu}

\author{Carla D. Savage}
\address{Department of Computer Science\\
         NC State University\\
         Raleigh, NC 27695-8206\\
         U.S.A.}
\email{savage@ncsu.edu}

\date{\today}
\begin{abstract}
It follows from work of Chung and Graham that for a certain family
of polynomials $T_{n}(x)$, derived from the descent statistic on permutations,
the coefficient sequence of $T_{n-1}(x)$ coincides with
that of the polynomial $T_{n}(x)/(1+x+\cdots+x^{n-1})$. We observed computationally that the inflated $\mathbf{s}$-Eulerian
polynomial $Q_{n}^{(\mathbf{s})}(x)$, which satisfies $Q_{n}^{(\mathbf{s})}(x) = T_{n}(x)$ when $\mathbf{s}=(1,2,\ldots,n)$, also satisfies this property for many sequences $\mathbf{s}$.
In this work we characterize those sequences $\mathbf{s}$ for which the coefficient sequence of $Q_{n-1}^{(\mathbf{s})}(x)$ coincides
with that of the polynomial $Q_{n}^{(\mathbf{s})}(x)/\left(1+x+\cdots+x^{s_{n}-1}\right)$.
In particular, we show that all nondecreasing sequences satisfy this
property.

We also  settle a conjecture of Pensyl and Savage by showing that the inflated {\bf s}-Eulerian polynomials are unimodal for
all choices of positive integer sequences ${\bf s}$.  In addition, we determine when these polynomials are palindromic and show our characterization is equivalent to another of Beck, Braun, K{\"o}ppe, Savage, and Zafeirakopoulos.
\end{abstract}



\maketitle

\section{\label{sec:Introduction}Introduction}

We present the solution to a problem on $\mathbf{s}$-lecture hall
partitions that was motivated by the work
of Chung and Graham~\cite{ChungGraham2013} on the maxdrop statistic in permutations.

Let $S_{n}$ be the set of permutations of $[n]=\{1,2,\ldots,n\}$.
Given a permutation $\pi=\pi_{1}\pi_{2}\ldots\pi_{n}$ in $S_{n}$,
define the set of \emph{descents} of $\pi$ to be $\Des\pi=\{i\in[n-1]\ |\ \pi_{i}>\pi_{i+1}\}$ and denote $\des\pi=|\Des\pi|$.
The following result is a consequence of Theorems 4.1 and 4.2 in~\cite{ChungGraham2013}.

\begin{prop}[Chung, Graham~\cite{ChungGraham2013}]\label{prop:ChungGraham}
Let $T_{n}(x)$ be defined by
\[
T_{n}(x)=\sum_{\pi\in S_{n}}x^{n\left(\des\pi-1\right)+\pi_{n}}.
\]
Then
\[
\frac{T_{n}(x)}{1+x+\cdots+x^{n-1}}=\sum_{\pi\in S_{n-1}}x^{n\left(\des\pi-1\right)+\pi_{n-1}+1}.
\]
\end{prop}

\begin{example}\label{exa:ChungGraham_coincide} Using Proposition~\ref{prop:ChungGraham}, we may compute:
\begin{align*}
T_{3}(x) &= \mathbf{1}x^{4}+\mathbf{1}x^{3}+\mathbf{2}x^{2}+\mathbf{1}x+1,\textnormal{ and}\\
\frac{T_{4}(x)}{1+x+x^{2}+x^{3}} &= \mathbf{1}x^{6}+\mathbf{1}x^{4}+\mathbf{2}x^{3}+\mathbf{1}x^{2}+1.
\end{align*}
\end{example}

By Example~\ref{exa:ChungGraham_coincide}, the coefficient sequences (i.e., the sequence of nonzero coefficients) of the polynomials $T_{3}(x)$ and $T_{4}(x)/\left(1+x+x^{2}+x^{3}\right)$ coincide. Chung and Graham~\cite{ChungGraham2013} use Proposition~\ref{prop:ChungGraham} to show that this holds for $n\geq 2$. Namely, they prove the following result.

\begin{cor}[Chung, Graham~\cite{ChungGraham2013}]\label{cor:ChungGraham} The coefficient sequence of the polynomial $T_{n-1}(x)$ coincides with that of ${T_{n}(x)/(1+x+\cdots+x^{n-1})}$.
\end{cor}

We will define the polynomial $\Qn(x)$, for positive integer sequences $\s=(s_{1},s_{2},\ldots,s_{n})$, of which $T_{n}(x)$ is a particular case. Indeed, we note that $T_{n}(x)=Q_{n}^{(1,2,\ldots,n)}(x)$ in Section~\ref{sec:Reversible-sequences}. We will show that
\[
\frac{Q_{n}^{(\mathbf{s})}(x)}{1+x+\cdots+x^{s_{n}-1}}
\]
is a polynomial and provide a combinatorial interpretation for it, see Theorem~\ref{ThmSavage}. In fact, Proposition~\ref{prop:ChungGraham} is a particular instance of Theorem~\ref{ThmSavage}. Furthermore, we will show that Corollary~\ref{cor:ChungGraham} is a special case of one of our main results, Theorem~\ref{thm:Main}, when $\mathbf{s}=(1,2,\ldots,n)$.

Given a sequence $\mathbf{s}=(s_{1},s_{2},\ldots,s_{n})$ of positive
integers, we define the $n$-dimensional $\mathbf{s}$-\emph{lecture hall cone}
\[
\Co_{n}^{(\mathbf{s})}=\left\{ \lambda\in\mathbb{R}^{n}\mid0\leq\frac{\lambda_{1}}{s_{1}}\leq\frac{\lambda_{2}}{s_{2}}\leq\dots\leq\frac{\lambda_{n}}{s_{n}}\right\} .
\]
The lattice points $\Co_{n}^{(\mathbf{s})}\cap\mathbb{Z}^{n}$ are called
$\mathbf{s}$-\emph{lecture hall partitions} (into $n$ parts). The cone $\Co_{n}^{(\s)}$ is generated by the vectors $\left\{ \bv_{i}=[0,\dots,0,s_{i},\dots,s_{n}]:1\leq i\leq n\right\} $.
The (half open) \emph{fundamental parallelepiped} associated to the generating set $V_{n}(\s)=\{\bv_{1},\bv_{2},\ldots,\bv_{n}\}$ is
\[
\Pi_{n}^{(\mathbf{s})}=\left\{ \sum_{i=1}^{n}\alpha_{i}\bv_{i}\mid0\leq\alpha_{i}<1\right\}.
\]

The generating function for the lattice points in $\Co_{n}^{(\mathbf{s})}$
can be computed from its fundamental parallelepiped and generators
as
\begin{equation}
\sum_{\la\in \Co_{n}^{(\mathbf{s})}\cap\mathbb{Z}^{n}}x^\la=\frac{\sum_{\la\in\Pi_{n}^{(\mathbf{s})}\cap\mathbb{Z}^{n}}x^{\la}}{\prod_{i=1}^{n}(1-x^{\bv_{i}})},\label{eqn:lpe}
\end{equation}
where $x^{\la}=x_{1}^{\la_{1}}x_{2}^{\la_{2}}\cdots x_{n}^{\la_{n}}$, see, for instance,~\cite[p.~40]{miller2007geometric}. That is
\[
\sum_{\la\in \Co_{n}^{(\mathbf{s})}\cap\mathbb{Z}^{n}}x_{1}^{\la_{1}}x_{2}^{\la_{2}}\cdots x_{n}^{\la_{n}}=\frac{\sum_{\la\in\Pi_{n}^{(\mathbf{s})}\cap\mathbb{Z}^{n}}x_{1}^{\la_{1}}x_{2}^{\la_{2}}\cdots x_{n}^{\la_{n}}}{\prod_{i=1}^{n}(1-x_{i}^{s_{i}}\cdots x_{n}^{s_{n}})}.
\]
Setting $x_{1}=\dots=x_{n-1}=1$ and $x_{n}=x$ gives
\begin{equation}
\sum_{\la\in \Co_{n}^{(\mathbf{s})}\cap\mathbb{Z}^{n}}x^{\la_{n}}=\frac{\sum_{\la\in\Pi_{n}^{(\mathbf{s})}\cap\mathbb{Z}^{n}}x^{\la_{n}}}{(1-x^{s_{n}})^{n}}.\label{eqnx:lpe}
\end{equation}

Define
\begin{equation}
Q_{n}^{(\mathbf{s})}(x)=\sum_{\lambda\in\Pi_{n}^{(\mathbf{s})}\cap\mathbb{Z}^{n}}x^{\lambda_{n}}. \label{eq:FirstDefQ}
\end{equation}
Corollary~\ref{cor:PensylSavage} shows that $Q_{n}^{(\mathbf{s})}(x)$ is the $n$th \emph{inflated} $\mathbf{s}$-\emph{Eulerian polynomial} associated to
$\Co_{n}^{(\mathbf{s})}$, which was introduced by Pensyl and Savage~\cite{PensylSavage2013}. For this reason, henceforth, we refer to $\Qn (x)$ as the $n$th inflated $\s$-Eulerian polynomial, or simply as the inflated $\s$-Eulerian polynomial because the index is clear.

We observed that for particular (infinite) sequences $\mathbf{s}=(s_{1},s_{2},\ldots)$
and positive integers $n$, the coefficient sequence of
$Q_{n-1}^{(s_1,s_2,\ldots,s_{n-1})}(x)$ coincides with that of the polynomial
\[
P_{n-1}^{(s_1,s_2,\ldots,s_n)}(x)=\frac{Q_{n}^{(s_1,s_2,\ldots,s_n)}(x)}{\ax}.
\]
Here we use the notation
\[
[a]_{y}=1+y+\cdots+y^{a-1},
\]
where $a$ is a positive integer and $y$ is a variable. This surprising fact led us to consider the problem of characterizing all (infinite) sequences $\mathbf{s}$ for which the respective coefficient sequences of these polynomials coincide, for all $n$. We call such sequences \emph{contractible}.

Henceforth, ${\mathbf{s}=(s_{1},s_{2},\ldots)}$ will denote an infinite sequence of positive integers, unless otherwise stated. In general, we abuse notation for the sake of convenience and write $\mathbf{s}$ in definitions reserved for finite sequences meaning the appropriately truncated version of $\mathbf{s}$. In particular, we may write $Q_{n}^{(\mathbf{s})}(x)$ and $P_{n}^{(\mathbf{s})}(x)$ in place of $Q_{n}^{(s_{1},s_{2},\ldots,s_{n})}(x)$  and $P_{n}^{(s_{1},s_{2},\ldots,s_{n})}(x)$, respectively.

\begin{example}\label{exa:fibonacci}Let $\mathbf{s}$ be the Fibonacci sequence, which is defined by $s_{1}=1$, $s_{2}=1$, and  $s_{i}=s_{i-1}+s_{i-2}$ for $i\geq 3$. Then
the $n$th inflated $\mathbf{s}$-Eulerian polynomials for the first
few $n$ are as follows:
\begin{align*}
Q_{5}^{(\s)}(x) & =\mathbf{1}x^{11}+\mathbf{1}x^{10}+\mathbf{2}x^{9}+\mathbf{4}x^{8}+\mathbf{4}x^{7}+\mathbf{4}x^{6}+\mathbf{4}x^{5}+\mathbf{4}x^{4}+\mathbf{2}x^{3}+\mathbf{2}x^{2}+\mathbf{1}x+\mathbf{1},\\
Q_{4}^{(\s)}(x) & =\mathbf{1}x^{4}+\mathbf{1}x^{3}+\mathbf{2}x^{2}+\mathbf{1}x+\mathbf{1},\\
Q_{3}^{(\s)}(x) & =\mathbf{1}x+\mathbf{1},\textrm{ and}\\
Q_{2}^{(\s)}(x) & =\mathbf{1}.
\end{align*}

On the other hand, the polynomials $P_{n}^{(\mathbf{s})}(x)$
for the first few $n$ are as follows:
\begin{align*}
P_{5}^{(\s)}(x) & =\mathbf{1}x^{18}+\mathbf{1}x^{16}+\mathbf{2}x^{15}+\mathbf{4}x^{13}+\mathbf{4}x^{12}+\mathbf{4}x^{10}+\mathbf{4}x^{8}+\mathbf{4}x^{7}+\mathbf{2}x^{5}+\mathbf{2}x^{4}+\mathbf{1}x^{2}+\mathbf{1},\\
P_{4}^{(\s)}(x) & =\mathbf{1}x^{7}+\mathbf{1}x^{5}+\mathbf{2}x^{4}+\mathbf{1}x^{2}+\mathbf{1},\\
P_{3}^{(\s)}(x) & =\mathbf{1}x^{2}+\mathbf{1},\textrm{ and}\\
P_{2}^{(\s)}(x) & =\mathbf{1}.
\end{align*}
\end{example}

We prove that all nondecreasing sequences are contractible.
Moreover, we characterize contractible sequences as follows.

\begin{thm}\label{thm:Main}Let $\mathbf{s}$ be an infinite sequence of positive integers. Then $\mathbf{s}$ is contractible if and only if either $(s_{i})_{i=3}^{\infty}$
is nondecreasing; or there exists $N\geq3$ such that $(s_{i})_{i=N}^{\infty}$
is nondecreasing, $s_{N}=s_{N-1}-1$ and $s_{j}=1$ for $j=1,2,\ldots,N-2$.\end{thm}

We use the tools that we develop for the proof of Theorem~\ref{thm:Main}---specifically, the combinatorial description of $P_{n}(x)$ given in Section~\ref{sec:Reversible-sequences}---to address the question of the unimodality and palindromicity of the inflated $\s$-Eulerian polynomial $\Qn(x)$, which was originally considered by Pensyl and Savage~\cite{PensylSavage2013}.

We say that a sequence $(a_0,a_1, \ldots, a_n)$ is {\it unimodal} if there is an integer $0 \leq t \leq n$, such that
\[
a_0 \leq a_1 \leq a_2 \leq \ldots \leq a_t > a_{t+1} \geq a_{t+2} \geq \ldots \geq a_n.
\]
A polynomial $P(x)=\sum_{k=0}^{n}a_{k}x^{k}$ is {\it unimodal} if its coefficient sequence $(a_{0},a_{1},\ldots,a_{n})$ is unimodal. For instance, we know from Example~\ref{exa:fibonacci} that
\begin{equation}\label{eq:Fibonacci_unimodal}
Q_{5}^{(1,1,2,3,5)}(x)=\mathbf{1}x^{11}+\mathbf{1}x^{10}+\mathbf{2}x^{9}+\mathbf{4}x^{8}+\mathbf{4}x^{7}+\mathbf{4}x^{6}+\mathbf{4}x^{5}+\mathbf{4}x^{4}+\mathbf{2}x^{3}+\mathbf{2}x^{2}+\mathbf{1}x+\mathbf{1}.
\end{equation}
Since the sequence $(1,2,2,4,4,4,4,4,2,1,1)$ is unimodal, the polynomial $Q_{5}^{(1,1,2,3,5)}(x)$ is unimodal. In fact, the polynomial $\Qn(x)$ is unimodal for any sequence of positive integers $\s$. Indeed, we prove the following result, which was conjectured by Pensyl and Savage~\cite[Sec.~5]{PensylSavage2013}, in Section~\ref{sec:unimodality}.

\begin{thm}\label{ThmUnimodal}
 For any sequence ${\bf s} = (s_1, s_2, \ldots )$ of positive integers, the polynomial $Q_n^{({\bf s})}(x)$ is unimodal (i.e., the coefficient sequence of $Q_n^{({\bf s})}(x)$ is unimodal).
\end{thm}

A polynomial $P(x) = \sum_{k=0}^n a_k x^k$ is {\it palindromic} if $a_k = a_{n-k}$ for $0 \leq k \leq n$, where we assume that $a_n  > 0$ and $a_0 \neq 0$. It is clear from Equation~\eqref{eq:Fibonacci_unimodal} that $Q_{5}^{(1,1,2,3,5)}(x)$ is not palindromic. However, the polynomial $\Qn(x)$ may be palindromic. For instance, we know from Example~\ref{exa:fibonacci} that
\[
Q_{4}^{(1,1,2,3)}(x) = \mathbf{1}x^{4}+\mathbf{1}x^{3}+\mathbf{2}x^{2}+\mathbf{1}x+\mathbf{1},
\]
so $Q_{4}^{(1,1,2,3)}(x)$ is palindromic.

In Section~\ref{sec:palindromicity}, we address the question of the palindromicity of $\Qn(x)$. We provide a characterization of positive integer sequences $\s$ for which $\Qn(x)$ is palindromic, see Theorem~\ref{thm:palindromic}. In particular, this result implies that $T_{n}(x)=Q_{n}^{(1,2,\ldots,n)}(x)$ is palindromic, as mentioned by Chung and Graham~\cite{ChungGraham2013}.

Beck, Braun, K{\"o}ppe, Savage, and Zafeirakopoulos~\cite[Thm.~4.1]{GorensteinPaper} provided a different characterization of palindromic $\Qn(x)$. We give a direct proof of the equivalence of the two characterizations in Section~\ref{sec:Gorenstein_cond_equiv}.

The rest of the paper is organized as follows. In Section~\ref{sec:background}, we provide historical background about $\s$-lecture hall partitions and the polynomial $\Qn (x)$ as well as some examples. Section~\ref{sec:A-alternative-description} presents a description of $Q_{n}^{(\mathbf{s})}(x)$ in terms of certain statistics on $\mathbf{s}$-inversion sequences, due to Pensyl and Savage~\cite{PensylSavage2013}. We use this result in Section~\ref{sec:Reversible-sequences} to provide a combinatorial characterization
of $P_{n}^{(\mathbf{s})}(x)$ in terms of \mbox{$\s$-inversion} sequences. In Section~\ref{sec:The-case-of},
we study the notion of contractibility in the light of the aforementioned combinatorial descriptions of $\Qn (x)$ and $P_{n}^{(\mathbf{s})}(x)$ in terms of $\s$-inversion sequences. We introduce a lemma (Lemma~\ref{lem:LemKey}), which is the main ingredient in the proof of Theorem~\ref{thm:Main}. We then use this lemma to show that all nondecreasing sequences are contractible. In Section~\ref{sec:Proof-of-the},
we conclude the proof of Theorem~\ref{thm:Main}, completing our characterization of contractible sequences. The remaining portion of the paper is devoted to study the unimodality and palindromicity of $\Qn(x)$. Specifically, we prove Theorem~\ref{ThmUnimodal} in Section~\ref{sec:unimodality} and we provide our characterization of palindromicity of $\Qn(x)$ in Section~\ref{sec:palindromicity}. In Section~\ref{sec:Gorenstein_cond_equiv}, we describe the relationship between our characterization of palindromicity and that of Beck et al. Finally, in Section~\ref{sec:concluding_remarks}, we make some concluding remarks regarding possible extensions of our work.

\section{\label{sec:background}Background and examples: lecture hall partitions and $Q_{n}^{(\mathbf{s})}(x)$}

Lecture hall partitions were introduced in 1997 by Bousquet-M{\'e}lou
and Eriksson~\cite{BousquetI}. They defined a \emph{lecture hall
partition} into $n$ parts as a partition $\lambda=(\lambda_{1},\lambda_{2},\dots,\lambda_{n})$
satisfying the inequality
\begin{equation}
0\leq\frac{\lambda_{1}}{1}\leq\frac{\lambda_{2}}{2}\leq\cdots\leq\frac{\lambda_{n}}{n}.\label{eq:lecthallcondition}
\end{equation}
This condition ensures that $\lambda_{i}\leq\lambda_{i+1}$
for $1\leq i\leq n-1$. Furthermore, if we order the parts $\lambda_{i}$
of a lecture hall partition from left to right in a diagram (see Figure~\ref{fig:LectureHall} for an example), we obtain a figure resembling
a lecture hall. Indeed,~\eqref{eq:lecthallcondition} is a sufficient
condition to allow students ($A$, $B$, $C$ and $D$ in the picture)
in each row (part) to see the professor ($O$ in the picture). This
fact led to the name of these restricted partitions.

\begin{figure}[H]
\noindent \begin{centering}
\includegraphics[scale=0.65]{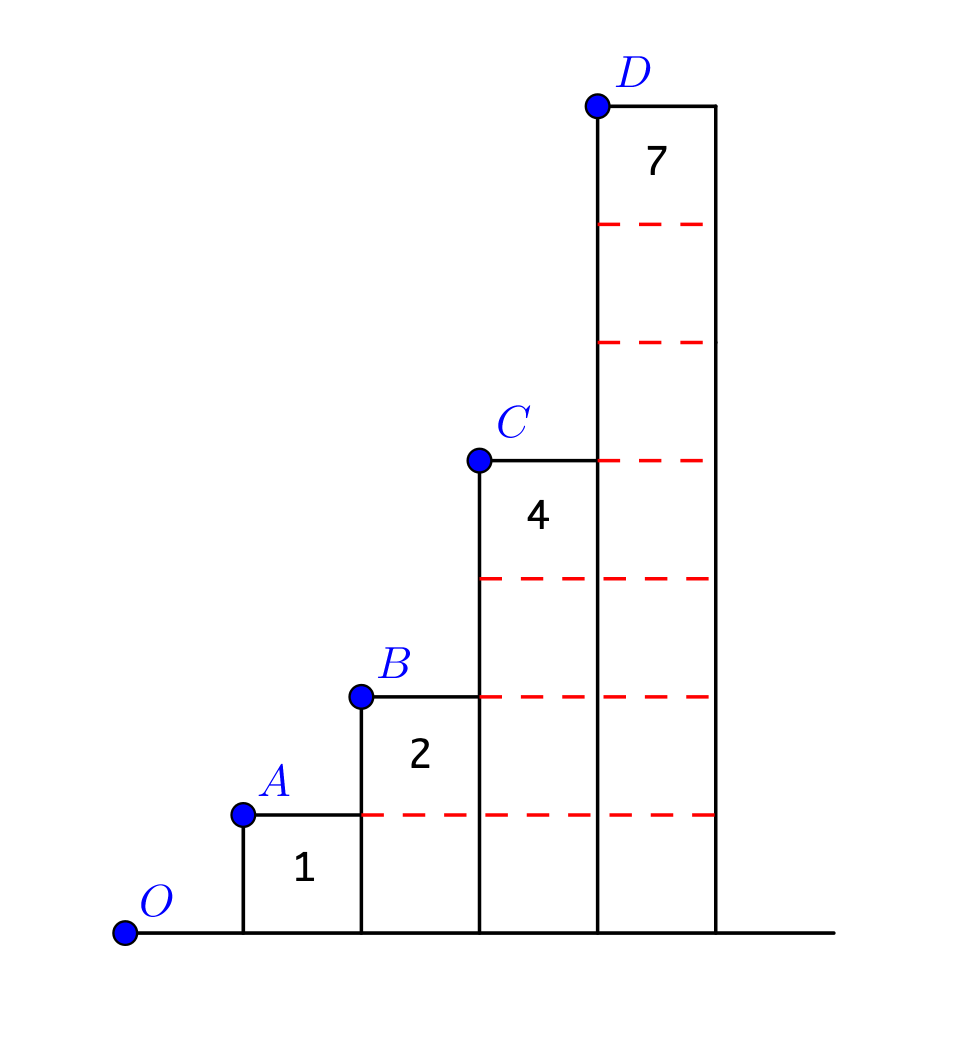}
\par\end{centering}

\protect\caption{Lecture hall partition $\lambda=(1,2,4,7)$ of $14$.\label{fig:LectureHall}}
\end{figure}

\begin{example}The partition $\lambda=(1,2,4,7)$ of $14$ is a lecture hall partition
because $\frac{1}{1}\leq\frac{2}{2}\leq\frac{4}{3}\leq\frac{7}{4}$.
A partition $\alpha=(\alpha_{1},\alpha_{2},\alpha_{3},\alpha_{4})$ such that $\alpha_{1}=1$, $\alpha_{2}=2$
and $\alpha_{3}=4$ must satisfy $\frac{4}{3}\leq\frac{\alpha_{4}}{4}$
in order to be a lecture hall partition, so necessarily, $\alpha_{4}\geq6$.
Thus, $\alpha=(1,2,4,5)$ is not a lecture hall partition. \end{example}

Lecture hall partitions quickly gained notoriety, because of the Lecture
Hall Theorem, proved by Bousquet-M{\'e}lou and Eriksson~\cite{BousquetI}.
This is probably the most remarkable result about lecture hall partitions. It relates lecture hall partitions to partitions into bounded
odd parts.

\begin{thm}[Bousquet-M{\'e}lou, Eriksson~\cite{BousquetI}]\label{thm:LectureHallThm}
For $n$ fixed,
the generating function for the number of lecture hall partitions
of $N$ into $n$ parts, $LH(N,n)$, is given by
\[
\sum_{N=0}^{\infty}LH(N,n)q^{N}=\prod_{i=1}^{n}\frac{1}{1-q^{2i-1}}.
\]
Thus, it coincides with the generating function of partitions
into odd parts, each less than $2n$. \end{thm}

Eriksen~\cite{eriksen2002simple} and Yee~\cite{Yee2001} provided bijective proofs of Theorem~\ref{thm:LectureHallThm}.

Since lecture hall partitions have distinct parts, we may think of
this result as a finite version of Euler's theorem asserting that
the number of partitions of $N$ into distinct parts coincides with
the number of partitions of $N$ into odd parts.

A natural generalization of lecture hall partitions is to consider
sequences $\la=(\la_{1},\la_{2},\ldots,\la_{n})$ of positive integers
such that
\begin{equation}
0\leq\frac{\lambda_{1}}{s_{1}}\leq\frac{\lambda_{2}}{s_{2}}\leq\dots\leq\frac{\lambda_{n}}{s_{n}}, \label{eq:DefOfsLectureHallPartitions}
\end{equation}
where $\mathbf{s}$ is a fixed sequence of positive
integers. These $\mathbf{s}$-lecture hall partitions were
introduced for nondecreasing $\s$ by Bousquet-M{\'e}lou and Eriksson~\cite{BousquetII} and for arbitrary positive
integer sequences by Savage and others~\cite{CorteelLeeSavage,CorteelSavage,CorteelSavageSills,SavageSchuster2012}.

\begin{example}Let $\mathbf{s}=(3,2)$. Then there are 6 $\mathbf{s}$-lecture
hall partitions in the fundamental parallelepiped $\Pi_{2}^{(\mathbf{s})}$.
Indeed, $\Pi_{2}^{(\mathbf{s})}\cap\mathbb{Z}^{2}=\left\{ (0,0),(0,1),(1,1),(1,2),(2,2),(2,3)\right\} $, see Figure~\ref{fig:Comparison}. In general, for a sequence $\mathbf{s}$ of positive integers, we have $\left|\Pi_{n}^{(\mathbf{s})}\cap\mathbb{Z}^{n}\right|=s_{1}s_{2}\cdots s_{n}$.
\end{example}

The name $\s$-lecture hall partitions is misleading because they need not be partitions in the traditional sense. For instance, $(1,2)$ and $(2,1)$ are points
of the $(5,2)$-lecture hall cone and, by definition, they are distinct
$(5,2)$-lecture hall partitions of 3, see Figure~\ref{fig:NonpartitionsExample}.
However, $(1,2)$ and $(2,1)$ are not distinct partitions. Nevertheless, if $\mathbf{s}$ is nondecreasing, then the $\s$-lecture hall partitions are in fact partitions. Indeed, if $\s$ is nondecreasing, then the inequalities in~\eqref{eq:DefOfsLectureHallPartitions} guarantee that  $\la$ is nondecreasing, so a nontrivial permutation of the entries of $\lambda$ cannot yield another $\s$-lecture hall partition. It is for this reason that $\s$-lecture hall partitions are of particular relevance when $\s$ in nondecreasing.

\begin{figure}[H]
\noindent \begin{centering}
\includegraphics[scale=0.8]{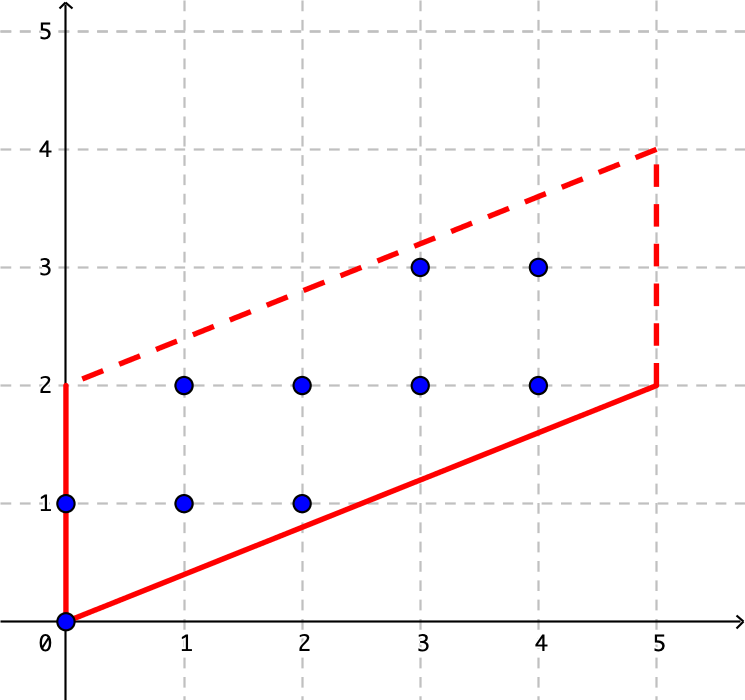}
\par\end{centering}
\protect\caption{Fundamental parallelepiped $\Pi_{2}^{(5,2)}$ of $\Co_{2}^{(5,2)}$.\label{fig:NonpartitionsExample}}
\end{figure}

The next example shows how to compute the inflated $\s$-Eulerian polynomial of a sequence $\s$ from the fundamental parallelepiped $\Pi_{n}^{(\s)}$ associated to the generating set $V_{n}(\s)$ of $\Co_{n}^{(\s)}$.

\begin{example}The sequence $\mathbf{s}=(5,3)$ has $\mathbf{s}$-lecture
hall cone $\Co_{2}^{(\mathbf{s})}=\left\{ (\lambda_{1},\lambda_{2})\in\mathbb{R}^{2}\mid0\leq\frac{\lambda_{1}}{5}\leq\frac{\lambda_{2}}{3}\right\} $.
This cone is generated by $\mathbf{v}_{1}=[5,3]$ and $\mathbf{v}_{2}=[0,3]$. There are $s_{1}s_{2}=15$
$\mathbf{s}$-lecture hall partitions in the fundamental
parallelepiped associated to these generators, see Figure~\ref{fig:Example}. Namely, the set $\Pi_{2}^{(\mathbf{s})}\cap\mathbb{Z}^{2}$
is described by
\[
\left\{ (0,0),(0,1),(1,1),(0,2),(1,2),(2,2),(3,2),(1,3),(2,3),(3,3),(4,3),(2,4),(3,4),(4,4),(4,5)\right\} .
\]
Therefore, we have the inflated $\mathbf{s}$-Eulerian polynomial
\[
Q_{2}^{(\mathbf{s})}(x)=\sum_{(\lambda_{1},\lambda_{2})\in\Pi_{2}^{(\mathbf{s})}\cap\mathbb{Z}^{2}}x^{\lambda_{2}} = x^{5}+3x^{4}+4x^{3}+4x^{2}+2x+1.
\]
\end{example}

\noindent \begin{center}
\begin{figure}[H]
\noindent \begin{centering}
\includegraphics[scale=0.85]{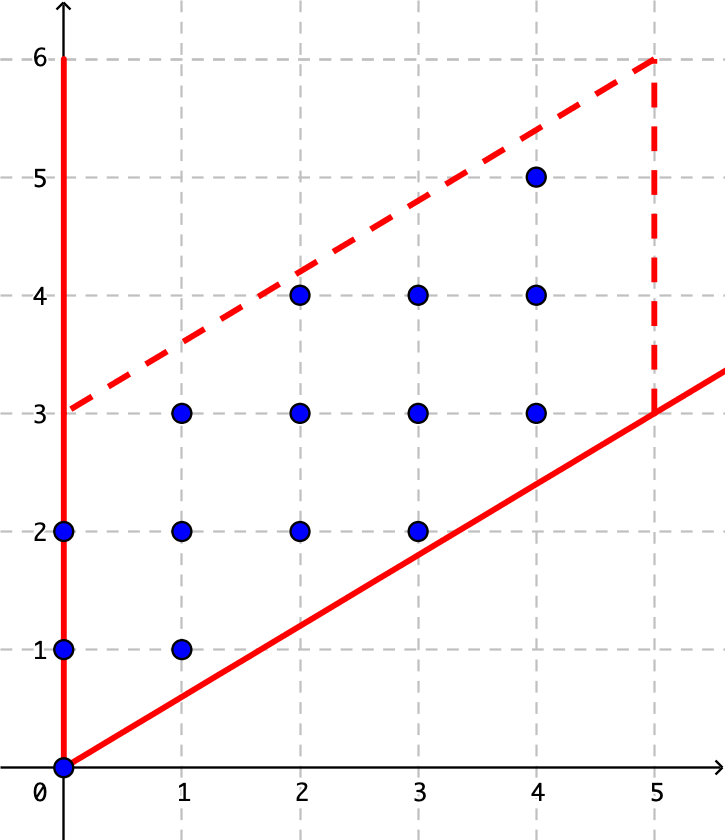}
\par\end{centering}

\protect\caption{The cone $\Co_{2}^{(5,3)}$ and its fundamental parallelepiped $\Pi_{2}^{(5,3)}$.\label{fig:Example}}
\end{figure}
\par\end{center}

If we replace the vector $\bv_{n}=[0,0,\ldots,s_{n}]$ by $\bv'_{n}=[0,0,\ldots,1]$ in $V_{n}(\s)$, we obtain the set $V'_{n}(\s)=\{\bv_{1},\ldots,\bv_{n-1},\bv'_{n}\}$, which also generates $\Co_{n}^{(\s)}$. However, the fundamental parallelepiped
\[
\Pi'_{n}(\s)=\left\{ \alpha_{n}\bv'_{n}+\sum_{i=1}^{n-1}\alpha_{i}\bv_{i}\mid0\leq\alpha_{i}<1\right\},
\]
associated to $V'_{n}(\s)$, does not coincide with $\Pi_{n}^{(\s)}$ if $s_{n}\neq1$, see Figure~\ref{fig:Comparison} for an example. Henceforth, whenever we refer to the  fundamental parallelepiped of $\Co_{n}^{(\s)}$, we mean $\Pi_{n}^{(\s)}$, unless otherwise stated.

In Section~\ref{sec:Reversible-sequences}, we will prove that $P_{n-1}^{(\mathbf{s})}(x)$ is in fact a polynomial by providing a combinatorial interpretation for it in terms of inversion sequences. However, it is also possible to prove the polynomiality of $P_{n-1}^{(\mathbf{s})}(x)$ via lattice point enumeration. Indeed, reinterpreting Equation~\eqref{eqn:lpe} in terms of $V'_{n}(\s)$ and $\Pi'_{n}(\s)$ we deduce that
\[
\sum_{\la\in \Co_{n}^{(\mathbf{s})}\cap\mathbb{Z}^{n}}x^\la=\frac{\sum_{\la\in\Pi'_{n}{(\mathbf{s})}\cap\mathbb{Z}^{n}}x^{\la}}{\left(1-x^{\bv'_{n}}\right)\prod_{i=1}^{n-1}(1-x^{\bv_{i}})},
\]
so setting $x_{1}=\cdots=x_{n-1}=1$ and $x_{n}=x$, we may write
\[
\sum_{\la\in \Co_{n}^{(\mathbf{s})}\cap\mathbb{Z}^{n}}x^{\la_{n}}=\frac{\sum_{\la\in\Pi'_{n}{(\mathbf{s})}\cap\mathbb{Z}^{n}}x^{\la_{n}}}{\left(1-x\right)(1-x^{s_{n}})^{n-1}}.
\]
It follows from Equation~(\ref{eqnx:lpe}) that
\[
\frac{\Qn (x)}{(1-x^{s_{n}})^{n}}=\frac{\sum_{\la\in\Pi'_{n}{(\mathbf{s})}\cap\mathbb{Z}^{n}}x^{\la_{n}}}{\left(1-x\right)(1-x^{s_{n}})^{n-1}}.
\]
Therefore,
\begin{equation}
P_{n-1}^{(\s)}(x)=\frac{\Qn (x)}{\ax}=\sum_{\la\in\Pi'_{n}{(\mathbf{s})}\cap\mathbb{Z}^{n}}x^{\la_{n}}. \label{eq:GeomRealization}
\end{equation}
The lattice point enumeration argument above is due to Savage~\cite[p.~30]{Savage2016}.

Equation~\eqref{eq:GeomRealization} not only proves that $P_{n-1}^{(\s)}(x)$ is a polynomial, it shows that this polynomial has a geometric realization. Indeed, just like $\Qn (x)$ enumerates the $\s$-lecture hall partitions in the fundamental parallelepiped $\Pi_{n}^{(\s)}$, by height, $P_{n-1}^{(\s)}(x)$ enumerates the $\s$-lecture hall partitions in the fundamental parallelepiped $\Pi'_{n}(\s)$, also by height. The following example illustrates the difference between the geometric interpretations of the polynomials $\Qn (x)$ and $P_{n-1}^{(\s)}(x)$.

\noindent \begin{center}
\begin{figure}[H]
\noindent \begin{centering}
\includegraphics[scale=0.86]{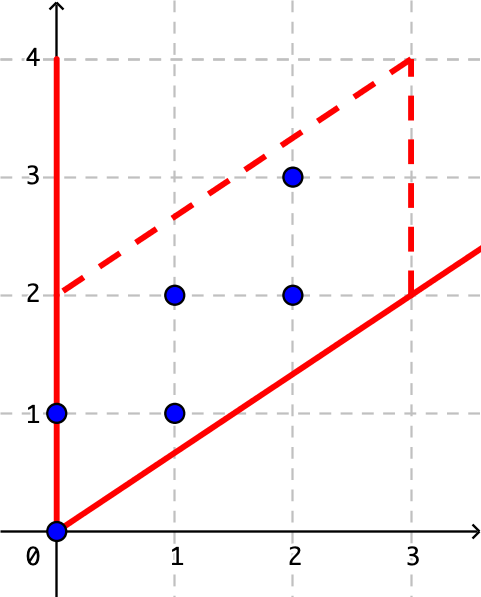}
\qquad{}\qquad{}\includegraphics[scale=0.86]{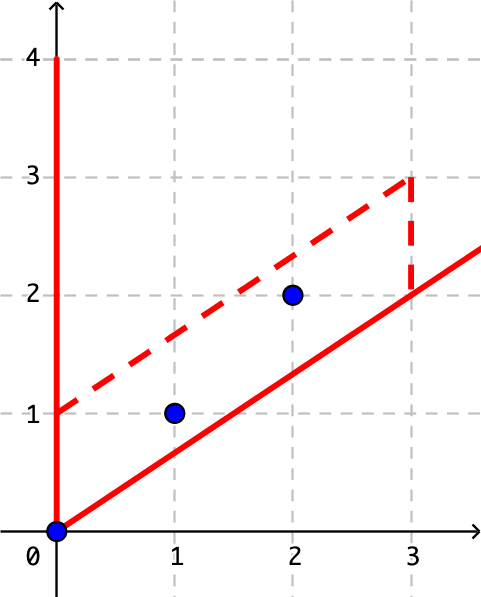}
\par\end{centering}

\protect\caption{Comparison between $\Pi_{2}^{(3,2)}$ and $\Pi'_{2}(3,2)$.\label{fig:Comparison}}
\end{figure}

\par\end{center}

\begin{example}Let $\mathbf{s}=(3,2)$. This sequence has $\mathbf{s}$-lecture hall cone
\[
\Co_{2}^{(\s)}=\left\{ (\lambda_{1},\lambda_{2})\in\mathbb{R}^{2}\mid0\leq\frac{\lambda_{1}}{3}\leq\frac{\lambda_{2}}{2}\right\}.
\]
It follows from the definition of the inflated $\s$-Eulerian polynomial and Figure~\ref{fig:Comparison} that
\[
 Q_{2}^{(\mathbf{s})}(x)=x^{3}+2x^{2}+2x+1.
\]
Thus, direct computation yields $P_{1}^{(\s)}(x)=Q_{2}^{(\s)}(x)/[s_{2}]_{x}=x^{2}+x+1$. On the other hand, we deduce from Figure~\ref{fig:Comparison} that $\Pi'_{2}(\s)\cap\Z^{2}=\left\{(0,0),(1,1),(2,2)\right\}$. Using Equation~\eqref{eq:GeomRealization}, we conclude that
\[
P_{1}^{(\s)}(x)=\sum_{\left(\la_{1},\la_{2}\right)\Pi'_{2}(\s)\cap\Z^{2}}x^{\la_{2}}=x^{2}+x+1,
\]
which, of course, coincides with our direct computation of $P_{1}^{(\s)}(x)$ from $Q_{2}^{(\s)}(x)$.
\end{example}

\section{\label{sec:A-alternative-description}Inversion sequences and $Q_{n}^{(\mathbf{s})}(x)$}

The definition of the inflated $\mathbf{s}$-Eulerian polynomial in Equation~\eqref{eq:FirstDefQ} is
very intuitive. However, we will need a different description of $Q_{n}^{(\mathbf{s})}(x)$ to characterize contractible sequences. In this section we introduce this alternative description, which is due to Pensyl and Savage~\cite{PensylSavage2013}.

Given a finite sequence of positive integers $\mathbf{s}=(s_{1},s_{2},\dots,s_{n})$, we define the $\mathbf{s}$-\emph{inversion
sequences} as
\[
I_{n}^{(\mathbf{s})}=\left\{ \e=(e_{1},e_{2},\dots,e_{n})\in\mathbb{Z}^{n}\mid0\leq e_{i}<s_{i}\textrm{ for }1\leq i\leq n\right\} .
\]
If $\e\in I_{n}^{(\mathbf{s})}$, we say $1\leq i<n$ is an \emph{ascent}
of $\e$ if $\frac{e_{i}}{s_{i}}<\frac{e_{i+1}}{s_{i+1}}$. We say
$0$ is an ascent of $\e$ if $e_{1}>0$. It is customary to denote the set of ascents of $\e\in I_{n}^{(\mathbf{s})}$
by $\Asc\mathbf{e}$ and its cardinality by $\asc\mathbf{e}$.
We denote the collection $\left\{ \Asc\mathbf{e}\mid\e\in I_{n}^{(\mathbf{s})}\right\} $
by $\Asc_{n}^{(\mathbf{s})}$.

\begin{example}Consider the sequence $\s=(1,1,2,2)$. By definition,
the $\s$-inversion sequences are given by $I_{4}^{(\s)}=\left\{ (0,0,0,0),(0,0,0,1),(0,0,1,0),(0,0,1,1)\right\} $.
Note that $3\notin\Asc(0,0,1,1)$ because $\frac{1}{2}\nless\frac{1}{2}$.
However, the inequality $\frac{0}{1}<\frac{1}{2}$ implies that $2\in\Asc(0,0,1,1)$.
\end{example}

We now state Pensyl and Savage's result. It leads to a description
of $Q_{n}^{(\mathbf{s})}(x)$ in terms of the ascents of the $\mathbf{s}$-inversion
sequences $I_{n}^{(\mathbf{s})}$.

\begin{thm}[Pensyl, Savage~\cite{PensylSavage2013}]\label{ThmPensylSavage}
Let $\mathbf{s}$ be a sequence of positive integers. Then
\[
\sum_{\la\in \Co_{n}^{(\mathbf{s})}\cap\mathbb{Z}^{n}}x^{\la_{n}}=\frac{\sum_{\e\in I_{n}^{(\mathbf{s})}}x^{\asc\mathbf{e}-e_{n}}}{(1-x^{s_{n}})^{n}}.
\]
\end{thm}

Combining Theorem~\ref{ThmPensylSavage} and Equation~(\ref{eqnx:lpe}) we
obtain the following corollary.

\begin{cor}[Pensyl, Savage~\cite{PensylSavage2013}]\label{cor:PensylSavage} Let $\s$ be a sequence of
positive integers, then
\[
Q_{n}^{(\mathbf{s})}(x)=\sum_{\mathbf{e}\in I_{n}^{(\mathbf{s})}}x^{s_{n}\asc\mathbf{e}-e_{n}}.
\]
\end{cor}

This result is analogous to the combinatorial interpretation of the $n$th $\s$-Eulerian polynomial given by Savage and Schuster~\cite{SavageSchuster2012}.

In Section~\ref{sec:Reversible-sequences}, we will use Corollary~\ref{cor:PensylSavage} to describe $P_{n}^{(\s)}(x)$ in terms of statistics of $\s$-inversion sequences.

\begin{example}Consider the sequence $\mathbf{s}=(1,2,3)$. The
$\mathbf{s}$-inversion sequences
\[
I_{3}^{(\s)}=\left\{ (0,0,0),(0,0,1),(0,0,2),(0,1,0),(0,1,1),(0,1,2)\right\}
\]
have ascents $\Asc_{3}^{(\s)}=\left\{ \emptyset,\{2\},\{2\},\{1\},\{1\},\{1,2\}\right\} $,
respectively. By Corollary~\ref{cor:PensylSavage},
\begin{align*}
Q_{3}^{(\s)}(x) & =x^{3(0)-0}+x^{3(1)-1}+x^{3(1)-2}+x^{3(1)-0}+x^{3(1)-1}+x^{3(2)-2}\\
 & =x^{4}+x^{3}+2x^{2}+x+1.
\end{align*}
\end{example}

\section{\label{sec:Reversible-sequences}Contractible sequences}

In this section, we engage the notion of contractibility for a
sequence $\mathbf{s}$ of positive integers. We begin by proving a very convenient combinatorial characterization of $P_{n-1}^{(\mathbf{s})}(x)$. This characterization is analogous to that of $\Qn (x)$ provided by Theorem~\ref{ThmPensylSavage}.

Given a sequence $\mathbf{e}=(e_{1},e_{2},\ldots,e_{n-1})\in I_{n-1}$
and $0\leq k<s_{n}$, we denote $(e_{1,}e_{2},\ldots,e_{n-1},k)\in I_{n}$
by $(\mathbf{e},k)$.

\begin{thm}\label{ThmSavage} Let $\mathbf{s}$ be a sequence of
positive integers, then
\[
P_{n-1}^{(\mathbf{s})}(x)=\sum_{\mathbf{e}\in I_{n-1}^{(\mathbf{s})}}x^{s_{n}\asc\mathbf{e}-\left\lfloor \frac{s_{n}e_{n-1}}{s_{n-1}}\right\rfloor }.
\]
\end{thm} \begin{proof} Let $\mathbf{s}$ be a sequence of positive
integers and $\mathbf{e}\in I_{n-1}^{(\mathbf{s})}$. Let $m$ be
the smallest positive integer such that $\frac{e_{n-1}}{s_{n-1}}<\frac{m}{s_{n}}$.
Then $m-1\leq\frac{s_{n}e_{n-1}}{s_{n-1}}<m$ and consequently, $\left\lfloor \frac{s_{n}e_{n-1}}{s_{n-1}}\right\rfloor =m-1$.
Note that $\frac{e_{n-1}}{s_{n-1}}<1$ implies that $m\leq s_{n}$.
Write
\begin{align}
\left(\sum_{k=0}^{s_{n}-1}x^{k}\right)x^{s_{n}\asc\mathbf{e}-\left\lfloor \frac{s_{n}e_{n-1}}{s_{n-1}}\right\rfloor } & =\left(\sum_{k=0}^{s_{n}-1}x^{k}\right)x^{s_{n}\asc\mathbf{e}-m+1}\nonumber \\
 & =\sum_{k=0}^{m-1}x^{s_{n}\asc\mathbf{e}-(m-k-1)}+\sum_{k=m}^{s_{n}-1}x^{s_{n}\asc\mathbf{e}-(m-k-1)}.\label{eq:mid}
\end{align}
If $m=s_{n}$, the rightmost sum is empty. Since $\frac{e_{n-1}}{s_{n-1}}\geq\frac{k}{s_{n}}$
for $0\leq k\leq m-1$, we know that
\begin{align*}
\sum_{k=0}^{m-1}x^{s_{n}\asc\mathbf{e}-(m-k-1)} & =x^{s_{n}\asc\mathbf{e}-0}+x^{s_{n}\asc\mathbf{e}-1}+\cdots+x^{s_{n}\asc\mathbf{e}-(m-1)}\\
 & =x^{s_{n}\asc(\mathbf{e},0)-0}+x^{s_{n}\asc(\mathbf{e},1)-1}+\cdots+x^{s_{n}\asc(\mathbf{e},m-1)-(m-1)}.
\end{align*}

Similarly, the fact that $\frac{e_{n-1}}{s_{n-1}}<\frac{k}{s_{n}}$
for $m\leq k\leq s_{n}-1$ implies that
\begin{align*}
\sum_{k=m}^{s_{n}-1}x^{s_{n}\asc\mathbf{e}-(m-k-1)} & =x^{s_{n}\asc\mathbf{e}+1}+x^{s_{n}\asc\mathbf{e}+2}+\cdots+x^{s_{n}\asc\mathbf{e}+s_{n}-m}\\
 & =x^{s_{n}(1+\asc\mathbf{e})-(s_{n}-1)}+x^{s_{n}(1+\asc\mathbf{e})-(s_{n}-2)}+\cdots+x^{s_{n}(1+\asc\mathbf{e})-m}\\
 & =x^{s_{n}\asc(\mathbf{e},s_{n}-1)-(s_{n}-1)}+x^{s_{n}\asc(\mathbf{e},s_{n}-2)-(s_{n}-2)}+\cdots+x^{s_{n}\asc(\mathbf{e},m)-m}.
\end{align*}

By Equation~\eqref{eq:mid}, we deduce that for each $\mathbf{e}\in I_{n-1}^{(\mathbf{s})}$,
\[
\left(\sum_{k=0}^{s_{n}-1}x^{k}\right)x^{s_{n}\asc\mathbf{e}-\left\lfloor \frac{s_{n}e_{n-1}}{s_{n-1}}\right\rfloor }=\sum_{k=0}^{s_{n}-1}x^{s_{n}\asc(\mathbf{e},k)-k},
\]
so we may write
\[
\left(\sum_{k=0}^{s_{n}-1}x^{k}\right)\sum_{\mathbf{e}\in I_{n-1}^{(\mathbf{s})}}x^{s_{n}\asc\mathbf{e}-\left\lfloor \frac{s_{n}e_{n-1}}{s_{n-1}}\right\rfloor }=\sum_{\mathbf{e}\in I_{n}^{(\mathbf{s})}}x^{s_{n}\asc\mathbf{e}-e_{n}}=Q_{n}^{(\mathbf{s})}(x).
\]
The rightmost equality follows from Corollary~\ref{cor:PensylSavage}.
\end{proof}

In particular, Theorem~\ref{ThmSavage} provides an alternative proof of the polynomiality of $P_{n-1}^{(\s)}(x)$. Of course, we already knew this from Equation~\eqref{eq:GeomRealization}. However, Theorem~\ref{ThmSavage} is novel in that it gives a description of $P_{n-1}^{(\s)}(x)$ in terms of $\s$-inversion sequences, just like Theorem~\ref{ThmPensylSavage} does for $\Qn (x)$. These combinatorial characterizations of $\Qn (x)$ and $P_{n-1}^{(\s)}(x)$ will be exploited in Section~\ref{sec:The-case-of} to prove that nondecreasing sequences are contractible and then again in Section~\ref{sec:Proof-of-the} to prove Theorem~\ref{thm:Main}.

The following corollary is an immediate consequence of Theorem~\ref{ThmSavage}.

\begin{cor}[Chung, Graham~\cite{ChungGraham2013}]\label{CorChungGraham}
\[
P_{n-1}^{(1,2,\dots,n)}(x)=\frac{Q_{n}^{(1,2,\dots,n)}(x)}{\anx}=\sum_{\mathbf{e}\in I_{n-1}}x^{n\cdot\asc\mathbf{e}-e_{n-1}}.
\]
\end{cor}

In fact, Corollary~\ref{CorChungGraham} is equivalent to Proposition~\ref{prop:ChungGraham}. Indeed, to see this consider the mapping $\phi:S_{n}\rightarrow I_{n}$
defined by $\phi(\pi)=(e_{1},\ldots,e_{n})$, where $e_{i}=|\{j>0\ |\ j<i\ {\rm and}\ \pi_{j}>\pi_{i}\}|$. Then $\Des\pi=\Asc\phi(\pi)$ and $e_{n}=n-\pi_{n}$. This shows that $Q_{n}^{(1,2,\ldots,n)}(x)=T_{n}(x)$ and $P_{n-1}^{(1,2,\dots,n)}(x)=T_{n}(x)/\anx$, so the equivalence between Proposition~\ref{prop:ChungGraham} and Corollary~\ref{CorChungGraham} follows.

We mentioned in Section~\ref{sec:Introduction} that the coefficient sequence of $T_{n-1}(x)$ coincides with that of the polynomial $T_{n}(x)/\anx$. Hence, $\s=(1,2,3,\ldots)$ is a contractible sequence. Corollaries~\ref{cor:PensylSavage} and~\ref{CorChungGraham} confirm this fact. Computational trials show that it is common among positive sequences $\mathbf{s}$ that the coefficient sequences of $\Qnn(x)$
and $P_{n-1}^{(\s)}(x)$ coincide, at least for small $n$. This
fact motivates the definition of $n$\emph{-contractible} sequences. If $n\geq3$, we say a positive sequence $\mathbf{s}$
is $n$-contractible if the coefficient sequences
of $Q_{n-1}^{(\mathbf{s})}(x)$ and $P_{n-1}^{(\s)}(x)$ coincide. Therefore, a contractible sequence is one that is $n$-contractible for $n\geq3$.

\begin{example}Let $\mathbf{s}$ be the Fibonacci sequence. Using
Corollaries~\ref{cor:PensylSavage} and~\ref{CorChungGraham}, it is possible to compute $Q_{n-1}^{(\mathbf{s})}(x)$
and $P_{n-1}^{(\s)}(x)$ for the first few $n$, see Example~\ref{exa:fibonacci}, and verify that the
Fibonacci sequence is $n$-contractible for $n=3,4,5,6$.\end{example}

The following corollary shows that all positive constant sequences
are contractible. Furthermore, the polynomials $Q_{n-1}^{(\mathbf{s})}(x)$
and $P_{n-1}^{(\s)}(x)$ coincide if $\mathbf{s}$ is constant.

\begin{cor}\label{CorConstantCase} Let $\mathbf{s}$ be a positive
constant sequence. Then for all $n\geq2$,
\[
P_{n-1}^{(\s)}(x)=\frac{Q_{n}^{(\mathbf{s})}(x)}{\ax}=Q_{n-1}^{(\mathbf{s})}(x).
\]
\end{cor} \begin{proof} Corollary~\ref{cor:PensylSavage} and Theorem~\ref{ThmSavage} imply that
\[
P_{n-1}^{(\s)}(x)=\frac{Q_{n}^{(\mathbf{s})}(x)}{\ax}=\sum_{\mathbf{e}\in I_{n-1}^{(\mathbf{s})}}x^{s_{n}\asc\mathbf{e}-\left\lfloor \frac{s_{n}e_{n-1}}{s_{n-1}}\right\rfloor }=\sum_{\mathbf{e}\in I_{n-1}^{(\mathbf{s})}}x^{s_{n-1}\asc\mathbf{e}-e_{n-1}}=Q_{n-1}^{(\mathbf{s})}(x). \qedhere
\]
\end{proof}

\section{\label{sec:The-case-of}The case of nondecreasing sequences}

In this section, we prove that all nondecreasing sequences $\s$ are contractible. Our work relies on the combinatorial characterizations of $\Qn (x)$ and $P_{n}^{(\s)}(x)$ introduced in Sections~\ref{sec:A-alternative-description}~and~\ref{sec:Reversible-sequences}, respectively. As we mentioned before, the case of nondecreasing $\s$ is of particular interest because
for such $\mathbf{s}$ the $\s$-lecture hall partitions are partitions in the traditional sense.

We begin by showing that if $\e$ and $\bar{\e}$ are $\s$-inversion sequences in $\Inn$, then the relative order of their corresponding exponents in $\Qnn (x)$ is the same as that of their corresponding exponents in $P_{n-1}^{(\s)}(x)$. By Corollary~\ref{cor:PensylSavage} and Theorem~\ref{ThmSavage}, this is amounts to proving the following lemma.

\begin{lem}\label{lem:1} Let $\mathbf{s}$ be a sequence of positive
integers and $n\geq3$. Suppose that $\mathbf{e},\bar{\mathbf{e}}\in I_{n-1}^{(\mathbf{s})}$ are such that
\[
s_{n-1}\asc\mathbf{e}-e_{n-1}\leq s_{n-1}\asc\bar{\mathbf{e}}-\bar{e}_{n-1}.
\]
Then
\begin{equation}\label{eq:double_ref}
s_{n}\asc\mathbf{e}-\left\lfloor \frac{s_{n}e_{n-1}}{s_{n-1}}\right\rfloor \leq s_{n}\asc\bar{\mathbf{e}}-\left\lfloor \frac{s_{n}\bar{e}_{n-1}}{s_{n-1}}\right\rfloor .
\end{equation}
\end{lem}
\begin{proof}
Let $\mathbf{e},\bar{\mathbf{e}}\in I_{n-1}^{(\mathbf{s})}$ be such that $s_{n-1}\asc\mathbf{e}-e_{n-1}\leq s_{n-1}\asc\bar{\mathbf{e}}-\bar{e}_{n-1}$. First, we prove that $\asc\e\leq\asc \bar{\e}$. We proceed by contradiction. If $\asc\e>\asc\bar{\e}$, then $\asc\e-\asc\bar{\e}\geq 1$. Hence,
\[
s_{n-1}\left(\asc\e-\asc\bar{\e}\right)\geq s_{n-1}>e_{n-1}\geq e_{n-1}-\bar{e}_{n-1}.
\]
It follows that $s_{n-1}\asc\e-e_{n-1}>s_{n-1}\asc\bar{\e}-\bar{e}_{n-1}$, which is a contradiction. Thus, $\asc\e\leq\asc\bar{\e}$.

If $\asc\e=\asc\bar{\e}$, then $s_{n-1}\asc\mathbf{e}-e_{n-1}\leq s_{n-1}\asc\bar{\mathbf{e}}-\bar{e}_{n-1}$ implies that $e_{n-1}\geq\bar{e}_{n-1}$. This means that $\left\lfloor\frac{s_{n}e_{n-1}}{s_{n-1}}\right\rfloor\geq\left\lfloor\frac{s_{n}\bar{e}_{n-1}}{s_{n-1}}\right\rfloor$, so~\eqref{eq:double_ref} holds.

Suppose that $\asc\e<\asc\bar{\e}$. Then $1\leq\asc\bar{\e}-\asc\e$. This implies that $s_{n}\leq s_{n}\left(\asc\bar{\e}-\asc\e\right)$. Given that $\bar{e}_{n-1}<s_{n-1}$, we know that $\frac{s_{n}\bar{e}_{n-1}}{s_{n-1}}<s_{n}$ and consequently, $\left\lfloor\frac{s_{n}\bar{e}_{n-1}}{s_{n-1}}\right\rfloor<s_{n}$. We deduce that
\[
\left\lfloor\frac{s_{n}\bar{e}_{n-1}}{s_{n-1}}\right\rfloor-\left\lfloor\frac{s_{n}e_{n-1}}{s_{n-1}}\right\rfloor\leq\left\lfloor\frac{s_{n}\bar{e}_{n-1}}{s_{n-1}}\right\rfloor<s_{n}\leq s_{n}\left(\asc\bar{\e}-\asc\e\right).
\]
Hence, in this case we also conclude that~\eqref{eq:double_ref} holds.
\end{proof}

The next corollary is an immediate consequence of Lemma~\ref{lem:1}.

\begin{cor}\label{cor:FixingCharacterizationContractible}
Let $\mathbf{s}$ be a sequence of positive
integers and $n\geq3$. Let $\mathbf{e},\bar{\mathbf{e}}\in\Inn$
be such that
\[
s_{n-1}\asc\mathbf{e}-e_{n-1}=s_{n-1}\asc\bar{\mathbf{e}}-\bar{e}_{n-1}.
\]
Then
\[
s_{n}\asc\mathbf{e}-\left\lfloor \frac{s_{n}e_{n-1}}{s_{n-1}}\right\rfloor = s_{n}\asc\bar{\mathbf{e}}-\left\lfloor \frac{s_{n}\bar{e}_{n-1}}{s_{n-1}}\right\rfloor .
\]
\end{cor}

Corollary~\ref{cor:FixingCharacterizationContractible} shows that if $\e$ and $\bar{\e}$ are $\s$-inversion sequences in $\Inn$ whose exponents in $\Qnn (x)$ coincide, then their exponents in $P_{n-1}^{(\s)}(x)$ also coincide.

We now use Lemma~\ref{lem:1} and Corollary~\ref{cor:FixingCharacterizationContractible} to provide a characterization of $n$-contractible sequences. Namely, we show that a sequence $\s$ is $n$-contractible if and only if every pair of $\s$-inversion sequences $\e,\bar{\e}\in\Inn$ satisfy the converse of Corollary~\ref{cor:FixingCharacterizationContractible}.

\begin{lem}\label{lem:LemKey}Let $\mathbf{s}$ be a sequence of
positive integers and $n\geq3$. Then $\mathbf{s}$ is $n$-contractible
if and only if whenever $\mathbf{e},\bar{\mathbf{e}}\in I_{n-1}^{(\mathbf{s})}$
are such that
\[
s_{n}\asc\mathbf{e}-\left\lfloor \frac{s_{n}e_{n-1}}{s_{n-1}}\right\rfloor =s_{n}\asc\bar{\mathbf{e}}-\left\lfloor \frac{s_{n}\bar{e}_{n-1}}{s_{n-1}}\right\rfloor ,
\]
then
\[
s_{n-1}\asc\mathbf{e}-e_{n-1}=s_{n-1}\asc\bar{\mathbf{e}}-\bar{e}_{n-1}.
\]
\end{lem}
\begin{proof}
Let $\s$ be a sequence of positive integers and $n\geq3$. Let $\e,\bar{\e}\in\Inn$ be such that
\[
s_{n}\asc\mathbf{e}-\left\lfloor \frac{s_{n}e_{n-1}}{s_{n-1}}\right\rfloor =s_{n}\asc\bar{\mathbf{e}}-\left\lfloor \frac{s_{n}\bar{e}_{n-1}}{s_{n-1}}\right\rfloor.
\]
and
\[
s_{n-1}\asc\mathbf{e}-e_{n-1}\neq s_{n-1}\asc\bar{\mathbf{e}}-\bar{e}_{n-1}.
\]
Say $\Qnn (x)$ has $k$ nonzero coefficients. If $\hat{\e}\in\Inn$ satisfies
\[
s_{n-1}\asc\mathbf{e}-e_{n-1}=s_{n-1}\asc\hat{\mathbf{e}}-\hat{e}_{n-1},
\]
then Corollary~\ref{cor:FixingCharacterizationContractible} implies that $s_{n}\asc\mathbf{e}-\left\lfloor \frac{s_{n}e_{n-1}}{s_{n-1}}\right\rfloor =s_{n}\asc\hat{\mathbf{e}}-\left\lfloor \frac{s_{n}\hat{e}_{n-1}}{s_{n-1}}\right\rfloor$. Similarly, if
\[
s_{n-1}\asc\hat{\e}-\hat{e}_{n-1}=s_{n-1}\asc\bar{\mathbf{e}}-\bar{e}_{n-1},
\]
then $s_{n}\asc\hat{\mathbf{e}}-\left\lfloor \frac{s_{n}\hat{e}_{n-1}}{s_{n-1}}\right\rfloor =s_{n}\asc\bar{\mathbf{e}}-\left\lfloor \frac{s_{n}\bar{e}_{n-1}}{s_{n-1}}\right\rfloor$. Using Corollary~\ref{cor:FixingCharacterizationContractible}, we deduce that $P_{n-1}^{(\s)}(x)$ has at most $k-1$ coefficients. In particular, this implies that the coefficient sequences of $\Qnn (x)$ and $P_{n-1}^{(\s)}(x)$ do not coincide, so $\s$ is not contractible.

Conversely, suppose that whenever $\e,\bar{\e}\in\Inn$ are such that
\[
s_{n}\asc\mathbf{e}-\left\lfloor \frac{s_{n}e_{n-1}}{s_{n-1}}\right\rfloor =s_{n}\asc\bar{\mathbf{e}}-\left\lfloor \frac{s_{n}\bar{e}_{n-1}}{s_{n-1}}\right\rfloor.
\]
Then $s_{n-1}\asc\mathbf{e}-e_{n-1}=s_{n-1}\asc\bar{\mathbf{e}}-\bar{e}_{n-1}$. By Corollary~\ref{cor:FixingCharacterizationContractible}, we deduce that if $\e,\bar{\e}\in\Inn$, then
\begin{equation}
s_{n-1}\asc\mathbf{e}-e_{n-1}=s_{n-1}\asc\bar{\mathbf{e}}-\bar{e}_{n-1}\quad\textrm{if and only if}\quad s_{n}\asc\mathbf{e}-\left\lfloor \frac{s_{n}e_{n-1}}{s_{n-1}}\right\rfloor =s_{n}\asc\bar{\mathbf{e}}-\left\lfloor \frac{s_{n}\bar{e}_{n-1}}{s_{n-1}}\right\rfloor.\label{eq:FailedCharacterization}
\end{equation}

Corollary~\ref{cor:PensylSavage} and Theorem~\ref{ThmSavage} imply that
\[
Q_{n-1}^{(\mathbf{s})}(x)=\sum_{e\in I_{n-1}^{(\mathbf{s})}}x^{s_{n-1}\asc e-e_{n-1}}\quad\textrm{and}\quad P_{n-1}^{(\s)}(x)=\sum_{e\in I_{n-1}^{(\mathbf{s})}}x^{s_{n}\asc e-\left\lfloor \frac{s_{n}e_{n-1}}{s_{n-1}}\right\rfloor }.
\]
Therefore,~\eqref{eq:FailedCharacterization} implies that the set of coefficients of $\Qnn (x)$ coincides with that of $P_{n-1}^{(\s)}(x)$. Furthermore, it follows from Lemma~\ref{lem:1} that the coefficient sequences of $\Qnn (x)$ and $P_{n-1}^{(\s)}(x)$ coincide. That is, $\s$ is $n$-contractible.
\end{proof}

We now show that $s_{n}\geq s_{n-1}$ implies that $\mathbf{s}$ is
$n$-contractible. This will imply the contractibility of nondecreasing sequences.

\begin{prop} \label{prop:1}Let $\mathbf{s}$ be a sequence of positive
integers and $n\geq3$. If $s_{n}\geq s_{n-1}$, then $\mathbf{s}$
is $n$-contractible. Thus, if $\s$ is such that $\left(s_{n}\right)_{n=3}^{\infty}$ is nondecreasing, then $\s$ is contractible. In particular, nondecreasing sequences are contractible.\end{prop}

\begin{proof} Let $\mathbf{s}$ be a sequence of positive
	integers and $n\geq3$. Suppose $s_{n}\geq s_{n-1}$. Let $\mathbf{e},\bar{\mathbf{e}}\in I_{n-1}^{(\mathbf{s})}$
be such that
\begin{equation}
s_{n}\asc\mathbf{e}-\left\lfloor \frac{s_{n}e_{n-1}}{s_{n-1}}\right\rfloor =s_{n}\asc\bar{\mathbf{e}}-\left\lfloor \frac{s_{n}\bar{e}_{n-1}}{s_{n-1}}\right\rfloor .\label{eq:FA}
\end{equation}
Since $0\leq e_{n-s},\bar{e}_{n-1}<s_{n-1}$, we know that the
inequalities $0\leq\frac{s_{n}e_{n-1}}{s_{n-1}},\frac{s_{n}\bar{e}_{n-1}}{s_{n-1}}<s_{n}$
hold, so
\[
0\leq\left\lfloor \frac{s_{n}e_{n-1}}{s_{n-1}}\right\rfloor ,\left\lfloor \frac{s_{n}\bar{e}_{n-1}}{s_{n-1}}\right\rfloor \leq s_{n}-1.
\]
Equation~\eqref{eq:FA} implies that $\left\lfloor \frac{s_{n}e_{n-1}}{s_{n-1}}\right\rfloor $
and $\left\lfloor \frac{s_{n}\bar{e}_{n-1}}{s_{n-1}}\right\rfloor $
are congruent modulo $s_{n}$, so we deduce that
\[
\left\lfloor \frac{s_{n}e_{n-1}}{s_{n-1}}\right\rfloor =\left\lfloor \frac{s_{n}\bar{e}_{n-1}}{s_{n-1}}\right\rfloor ,
\]
and therefore, $\asc\mathbf{e}=\asc\bar{\mathbf{e}}$.
Assume that $s_{n-1}\asc\mathbf{e}-e_{n-1}\neq s_{n-1}\asc\bar{\mathbf{e}}-\bar{e}_{n-1}$. Then it must be that $e_{n-1}\neq\bar{e}_{n-1}$. Say $e_{n-1}>\bar{e}_{n-1}$,
then
\[
\left\lfloor \frac{s_{n}\bar{e}_{n-1}}{s_{n-1}}\right\rfloor \leq\left\lfloor \frac{s_{n}(\bar{e}_{n-1}+1)}{s_{n-1}}\right\rfloor \leq\cdots\leq\left\lfloor \frac{s_{n}e_{n-1}}{s_{n-1}}\right\rfloor =\left\lfloor \frac{s_{n}\bar{e}_{n-1}}{s_{n-1}}\right\rfloor .
\]
Thus, $\left\lfloor \frac{s_{n}\bar{e}_{n-1}}{s_{n-1}}\right\rfloor =\left\lfloor \frac{s_{n}\bar{e}_{n-1}}{s_{n-1}}+\frac{s_{n}}{s_{n-1}}\right\rfloor \geq\left\lfloor \frac{s_{n}\bar{e}_{n-1}}{s_{n-1}}+1\right\rfloor =\left\lfloor \frac{s_{n}\bar{e}_{n-1}}{s_{n-1}}\right\rfloor +1$,
which is a contradiction. We conclude that $s_{n-1}\asc\mathbf{e}-e_{n-1}=s_{n-1}\asc\bar{\mathbf{e}}-\bar{e}_{n-1}$.
The result then follows from Lemma~\ref{lem:LemKey}.
\end{proof}

In the preceding proof the condition $s_{n}\geq s_{n-1}$
is not used to conclude that
\[
\left\lfloor \frac{s_{n}e_{n-1}}{s_{n-1}}\right\rfloor =\left\lfloor \frac{s_{n}\bar{e}_{n-1}}{s_{n-1}}\right\rfloor .
\]
Therefore, for any $\mathbf{s}$, if $\mathbf{e},\bar{\mathbf{e}}\in I_{n-1}^{(\mathbf{s})}$
are such that
\[
s_{n}\asc\mathbf{e}-\left\lfloor \frac{s_{n}e_{n-1}}{s_{n-1}}\right\rfloor =s_{n}\asc\bar{\mathbf{e}}-\left\lfloor \frac{s_{n}\bar{e}_{n-1}}{s_{n-1}}\right\rfloor ,
\]
then $\left\lfloor \frac{s_{n}e_{n-1}}{s_{n-1}}\right\rfloor =\left\lfloor \frac{s_{n}\bar{e}_{n-1}}{s_{n-1}}\right\rfloor $
and $\asc\mathbf{e}=\asc\bar{\mathbf{e}}$.

\section{\label{sec:Proof-of-the}A characterization of contractible sequences}

Although the criterion for contractibility provided by Proposition~\ref{prop:1} requires relatively weak conditions on $\mathbf{s}$,
there do exist noncontractible sequences. Indeed, the next example
exhibits an infinite family of noncontractible sequences.

\begin{example}Consider the finite sequence $(1,7,2)$. Using Corollary~\ref{cor:PensylSavage} and Theorem~\ref{ThmSavage}, we may compute:
\begin{align*}
Q_{2}^{(1,7)}(x) & =\mathbf{1}x^{6}+\mathbf{1}x^{5}+\mathbf{1}x^{4}+\mathbf{1}x^{3}+\mathbf{1}x^{2}+\mathbf{1}x+\mathbf{1},\textrm{ and}\\
P_{2}^{(1,7,2)}(x) & =\mathbf{3}x^{2}+\mathbf{3}x+\mathbf{1}.
\end{align*}
Therefore, any $\mathbf{s}$ such that $(s_{1},s_{2},s_{3})=(1,7,2)$
is not contractible.\end{example}

By Proposition~\ref{prop:1}, we know that if $\mathbf{s}$
is a sequence of positive integers such that $(s_{i})_{i=3}^{\infty}$
is nondecreasing, then $\mathbf{s}$ is contractible. The next example
shows that this sufficient criterion is not necessary.

\begin{example}Let $\mathbf{s}$ be the sequence defined by
\[
s_{n}=\begin{cases}
1 & \textrm{if }n\in\{1,2\},\\
3 & \textrm{if }n=3\textrm{ and},\\
2 & \textrm{if }n\geq4.
\end{cases}
\]
Using Corollary~\ref{cor:PensylSavage}, we find that
\begin{align*}
Q_{3}^{(\mathbf{s})}(x) & =\mathbf{1}x^{2}+\mathbf{1}x+\mathbf{1},\textrm{ and}\\
Q_{2}^{(\mathbf{s})}(x) & =\mathbf{1}.
\end{align*}
On the other hand, Theorem~\ref{ThmSavage} implies that
\begin{align*}
P_{3}^{(\mathbf{s})}(x) & =\mathbf{1}x^{2}+\mathbf{1}x+\mathbf{1},\textrm{ and}\\
P_{2}^{(\mathbf{s})}(x) & =\mathbf{1}.
\end{align*}
Suppose that $n>4$, then $s_{n}\geq s_{n-1}$ and it follows from
Proposition~\ref{prop:1} that $\mathbf{s}$ is $n$-contractible.
We see that although $s_{3}>s_{4}$, the sequence $\mathbf{s}$ is
contractible. \end{example}

We now conclude the proof of Theorem~\ref{thm:Main}. Our strategy is to exploit the characterization of \mbox{$n$-contractible} sequences provided by Lemma~\ref{lem:LemKey}. In order to do this, we need the following lemma.

\begin{lem}\label{lem:Prop}Let $\mathbf{s}$ be a sequence of positive
integers. Suppose that $n\geq3$ is such that $s_{n}\leq s_{n-1}-1$. If
there exists $1\leq j\leq n-2$ such that $s_{j}>1$, then $\mathbf{s}$
is not $n$-contractible.\end{lem}

\begin{proof} Since $s_{n}\leq s_{n-1}-1$, we know that $s_{n-1}\geq 2$. Thus, $\bar{\mathbf{e}}=(0,\ldots,0,1,0,\ldots0)$
($1$ in the $j$th entry) and $\mathbf{e}=(0,0,\ldots,0,1)$ are elements of $I_{n-1}^{(\mathbf{s})}$
such that $\asc\bar{\mathbf{e}}=1=\asc\mathbf{e}$
and $\bar{e}_{n-1}=0\neq1=e_{n-1}$. Now, $s_{n}\leq s_{n-1}-1$ implies that $\left\lfloor \frac{s_{n}e_{n-1}}{s_{n-1}}\right\rfloor =0=\left\lfloor \frac{s_{n}\bar{e}_{n-1}}{s_{n-1}}\right\rfloor $ and so,
\[
s_{n-1}\asc\mathbf{e}-e_{n-1}\neq s_{n-1}\asc\bar{\mathbf{e}}-\bar{e}_{n-1}\quad\textrm{and}\quad s_{n}\asc\mathbf{e}-\left\lfloor \frac{s_{n}e_{n-1}}{s_{n-1}}\right\rfloor =s_{n}\asc\bar{\mathbf{e}}-\left\lfloor \frac{s_{n}\bar{e}_{n-1}}{s_{n-1}}\right\rfloor .
\]
The result then follows from Lemma~\ref{lem:LemKey}.
\end{proof}

By Lemma~\ref{lem:Prop}, it only remains to consider sequences of
the form $(1,1,\ldots,1,s_{N-1},s_{N},\ldots)$, where $N\geq3$,
$s_{N}\leq s_{N-1}-1$ and $(s_{i})_{i=N}^{\infty}$ is nondecreasing.
Indeed, in order to conclude the proof of Theorem~\ref{thm:Main},
it suffices to show that if $n\geq3$ and $\mathbf{s}$ is a sequence
such that $s_{1}=s_{2}=\cdots=s_{n-2}=1$ and $s_{n}\leq s_{n-1}-1$,
then $\mathbf{s}$ is $n$-contractible if and only if $s_{n}=s_{n-1}-1$.
\begin{proof}[Proof of Theorem~\ref{thm:Main}]Let $n\geq3$ and suppose
$\mathbf{s}$ is a sequence such that $s_{1}=s_{2}=\cdots=s_{n-2}=1$
and $s_{n}\leq s_{n-1}-1$.

First, we prove that if $s_{n}=s_{n-1}-1$, then $\s$ is contractible. Assume $s_{n}=s_{n-1}-1$. Let $\mathbf{e},\bar{\mathbf{e}}\in I_{n-1}^{(\mathbf{s})}$
be such that $s_{n}\asc\mathbf{e}-\left\lfloor \frac{s_{n}e_{n-1}}{s_{n-1}}\right\rfloor =s_{n}\asc\bar{\mathbf{e}}-\left\lfloor \frac{s_{n}\bar{e}_{n-1}}{s_{n-1}}\right\rfloor $.
Then, by the remark following Proposition~\ref{prop:1}, we know that $\left\lfloor \frac{s_{n}e_{n-1}}{s_{n-1}}\right\rfloor =\left\lfloor \frac{s_{n}\bar{e}_{n-1}}{s_{n-1}}\right\rfloor $
and $\asc\mathbf{e}=\asc\bar{\mathbf{e}}$.
Write
\[
\left\lfloor \frac{s_{n}e_{n-1}}{s_{n-1}}\right\rfloor =\left\lfloor \frac{(s_{n-1}-1)e_{n-1}}{s_{n-1}}\right\rfloor =e_{n-1}+\left\lfloor \frac{-e_{n-1}}{s_{n-1}}\right\rfloor .
\]
This means that
\[
\left\lfloor \frac{s_{n}e_{n-1}}{s_{n-1}}\right\rfloor =\begin{cases}
0 & \textrm{if }e_{n-1}=0,\\
e_{n-1}-1 & \textrm{otherwise. }
\end{cases}
\]

Hence, if $e_{n-1}\neq\bar{e}_{n-1}$, then (without loss of
generality) $e_{n-1}=0$ and $\bar{e}_{n-1}=1$. This means that
$\mathbf{e}=(0,0,\ldots,0)$ and $\bar{\mathbf{e}}=(0,\ldots,0,1)$,
but then $\asc\mathbf{e}=0\neq1=\asc\bar{\mathbf{e}}$, which is a contradiction. We deduce that $e_{n-1}=\bar{e}_{n-1}$ and
consequently,
\[
s_{n-1}\asc\mathbf{e}-e_{n-1}=s_{n-1}\asc\bar{\mathbf{e}}-\bar{e}_{n-1}.
\]
Using Lemma~\ref{lem:LemKey}, we conclude that $\mathbf{s}$ is $n$-contractible
if $s_{n}=s_{n-1}-1$.

We now show that if $s_{n}<s_{n-1}-1$, then $\s$ is not contractible. Suppose $s_{n}<s_{n-1}-1$. The idea is to construct $\e,\bar{\e}\in\Inn$ such that
\[
s_{n-1}\asc\mathbf{e}-e_{n-1}\neq s_{n-1}\asc\bar{\mathbf{e}}-\bar{e}_{n-1}\quad\textrm{and}\quad s_{n}\asc\mathbf{e}-\left\lfloor \frac{s_{n}e_{n-1}}{s_{n-1}}\right\rfloor =s_{n}\asc\bar{\mathbf{e}}-\left\lfloor \frac{s_{n}\bar{e}_{n-1}}{s_{n-1}}\right\rfloor .
\]

Say $s_{n}=s_{n-1}-l$ for $2\leq l<s_{n-1}$. Note that $\left\lfloor \frac{-l(s_{n-1}-1)}{s_{n-1}}\right\rfloor =-l+\left\lfloor \frac{l}{s_{n-1}}\right\rfloor =-l\leq-2$.
Thus, we may choose $k$ to be the smallest integer in $\{2,3,\ldots,s_{n-1}-1\}$
such that $\left\lfloor \frac{-lk}{s_{n-1}}\right\rfloor \leq-2$.
Since $k-1\geq1$, it must be that $\left\lfloor \frac{-l(k-1)}{s_{n-1}}\right\rfloor =-1$. If $0\leq m<s_{n-1}$, then
\begin{equation}
\left\lfloor \frac{s_{n}m}{s_{n-1}}\right\rfloor =\left\lfloor \frac{(s_{n-1}-l)m}{s_{n-1}}\right\rfloor =m+\left\lfloor \frac{-lm}{s_{n-1}}\right\rfloor .\label{eq:1}
\end{equation}
Hence,
\begin{equation}
\left\lfloor \frac{s_{n}(k-1)}{s_{n-1}}\right\rfloor =(k-1)+(-1)=k-2.\label{eq:closing1}
\end{equation}

Now,
\[
\left\lfloor \frac{-lk}{s_{n-1}}\right\rfloor =\left\lfloor \frac{-l(k-1)}{s_{n-1}}-\frac{l}{s_{n-1}}\right\rfloor \geq\left\lfloor \frac{-l(k-1)}{s_{n-1}}\right\rfloor +\left\lfloor \frac{-l}{s_{n-1}}\right\rfloor =(-1)+(-1)=-2.
\]
Thus, it follows from Equation~\eqref{eq:1} that
\begin{equation}
\left\lfloor \frac{s_{n}k}{s_{n-1}}\right\rfloor =k-2.\label{eq:closing2}
\end{equation}

Define $\mathbf{e},\bar{\mathbf{e}}\in I_{n-1}^{(\mathbf{s})}$ by $\mathbf{e}=(0,\ldots,0,k)$ and $\bar{\mathbf{e}}=(0,\ldots,0,k-1)$, respectively. Then
\[
\asc\mathbf{e}=1=\asc\bar{\mathbf{e}}.
\]
Equations~\eqref{eq:closing1}~and~\eqref{eq:closing2} yield $\left\lfloor \frac{s_{n}e_{n-1}}{s_{n-1}}\right\rfloor =k-2=\left\lfloor \frac{s_{n}\bar{e}_{n-1}}{s_{n-1}}\right\rfloor $,
so
\[
s_{n}\asc\mathbf{e}-\left\lfloor \frac{s_{n}e_{n-1}}{s_{n-1}}\right\rfloor =s_{n}\asc\bar{\mathbf{e}}-\left\lfloor \frac{s_{n}\bar{e}_{n-1}}{s_{n-1}}\right\rfloor
\]
and, clearly, $s_{n-1}\asc\mathbf{e}-e_{n-1}\neq s_{n-1}\asc\bar{\mathbf{e}}-\bar{e}_{n-1}$. Therefore, using Lemma~\ref{lem:LemKey}, we deduce that $\mathbf{s}$ is not $n$-contractible
if $s_{n}<s_{n-1}-1$.
\end{proof}

\section{Unimodality of $Q_n^{({\bf s})}(x)$}\label{sec:unimodality}

Our goal in this section is to prove Theorem~\ref{ThmUnimodal}. That is, to prove that the polynomial $\Qn(x)$ is unimodal. First, we state a few elementary facts concerning unimodal sequences.

\subsection{Remarks on unimodal sequences.} For a given positive integer $n $, let $( a_1 \leq a_2 \leq a_3 \leq \dots \leq a_t > a_{t+1} \geq a_{t+2} \geq \dots \geq a_n )$ be a unimodal sequence of positive integers, where $t = n$ is possible. By convention, we set $a_j = 0$ if $j \leq 0$ or $j > n$. Fix a positive integer $m \leq n$
 and define $d_i = a_{i + m} - a_{i}$ for all $i$.
Note that $d_{1 - m} = a_1 > 0$ and $d_{n} =-a_n < 0$. We use the notation introduced above throughout this subsection.

\begin{lem}\label{lem:Ron_Fact1} Let $q$ be the least index such that $d_q  = a_{q +m} - a_q < 0$. Then $d_s \leq 0$ for all $s > q$.
\end{lem}

\begin{proof} Assume that $s > q$. Since $d_q = a_{q+m} - a_q <0$, it must be that $t < q$ so that $a_{q+m} \geq a_{s+m}$.
 Now, if $t \leq s$ then $a_s \geq a_{s+m}$ and we are done. Otherwise, $a_q \leq a_s$ and so $a_{s+m} -a_s \leq a_{q+m} -a_q = d_q < 0$, as claimed.
\end{proof}

The point of Lemma~\ref{lem:Ron_Fact1} is that the sequence $d_1, d_1 + d_2, d_1 +d_2 +d_3, \ldots, d_1 + d_2 + \ldots + d_n$ is always unimodal.

We next want to partition the index set $\mathbb{Z}$ into consecutive intervals of length $m$ as follows. Define $J_k = (km+1, km+2, \ldots, km+m)$ for
all $k$. Furthermore, for a positive integer $r \leq m$, we partition each $J_k$ into $r$ non-empty consecutive subintervals. Namely, define
$J_k = J_k(1) \cup J_k(2) \cup \dots \cup J_k(r)$ where $J_k(u)$ has length $l_u$ for some choice of $l_u > 0$ and $\sum_u  l_u = m$.
Let $s_k(u)$ denote the sum $\sum_{i \in J_k(u)} a_i$. Define the difference $D_k(u) = s_k(u)  - s_{k+1}(u)$ for $1 \leq u \leq r$ and all $k$. Let $k_0$ and $u_0$ be the indices such that $q \in J_{k_0}(u_0)$.

The next result is a straightforward consequence of Lemma~\ref{lem:Ron_Fact1}.

\begin{lem}\label{lem:Ron_Fact2} If $i < k_0$, then $D_i(u) \geq 0$, for all $u$; while if $i > k_0$, then $D_i(u) \leq 0$, for all $u$. Also,
if $u < u_0$, then $D_{k_0}(u) \geq 0$; while if $u > u_0$, then $D_{k_0}(u) \leq 0$.
\end{lem}

\subsection{Proof of Theorem~\ref{ThmUnimodal}}

\begin{proof}[Proof of Theorem~\ref{ThmUnimodal}]
From Theorem~\ref{ThmSavage}, we know that
\begin{align} \label{AS}
P_{n-1}^{({\bf s})}(x) = \frac{Q_n^{({\bf s})}(x)}{\ax} = \sum_{{\bf e} \in I_{n-1}^{({\bf s})}} x^{s_n \asc\e -\left \lfloor{\frac{s_n e_{n-1}}{s_{n-1}}} \right\rfloor }.
\end{align}
By Corollary\ref{cor:PensylSavage},
\begin{align}\label{l3}
Q_{n-1}^{({\bf s})}(x) = \sum_{{\bf e} \in I_{n-1}^{({\bf s})}} x^{s_{n-1} \asc\e - e_{n-1}}
\end{align}
In particular, if $s_{n-1} = s_n$ then $P_{n-1}^{({\bf s})}(x) = Q_{n-1}^{({\bf s})}(x).$

The basic idea in proving Theorem~\ref{ThmUnimodal} will be to
``bootstrap'' our way up from $Q_{n-1}^{({\bf s})}(x)$ to $P_{n-1}^{({\bf s})}(x)$ to $Q_n^{({\bf s})}(x)$, using Equations~\eqref{AS} and~\eqref{l3}, and induction on $n$. To start the induction, observe that for $n = 1$ and ${\bf s} = (s_1)$, we have $I_1^{(\s)} = \{0, 1, 2, \ldots, s_1 -1\}$. This implies, by definition of $\asc\e$, that
$Q_1^{(\s)}(x) = 1 + x + x^2+ \ldots + x^{s_1 -1}$,
which is certainly unimodal.

We will assume that $n \geq 2$ and that $Q_{n-1}^{({\bf s})}(x)$ is unimodal for every choice of ${\bf{s}} = (s_1, s_2, \ldots, s_n)$ with all $s_i > 0$. There are three cases.

\noindent{\bf Case 1.} ($s_{n-1} = s_n$): We have noted that in this case, $P_{n-1}^{({\bf s})}(x) = Q_{n-1}^{({\bf s})}(x).$ Thus,
by Equation~\eqref{AS},
\begin{align*}
    Q_n^{({\bf s})}(x) &= P_{n-1}^{({\bf s})}(x) \ax\\
    &= Q_{n-1}^{({\bf s})}(x)  \ax.
\end{align*}
It follows that $Q_{n}^{({\bf s})}(x)$ is unimodal, since $Q_{n-1}^{({\bf s})}(x)$ is unimodal by inductive hypothesis.

\noindent{\bf Case 2.} ($s_{n-1} < s_n$): By Proposition~\ref{prop:1} the polynomials $P_{n-1}^{({\bf s})}(x)$ and $Q_{n-1}^{({\bf s})}(x)$ have the same sequence of nonzero coefficients, the difference being that $P_{n-1}^{({\bf s})}(x)$ has occasional gaps in its terms. For example,
for the sequence ${\bf s} = (1,2,4,6)$, we have
\begin{align*}
Q_{3}^{({\bf s})}(x) &= x^5 + x^4 + 2x^3 + 2x^2 + x + 1, \quad\textnormal{and}\\
P_{3}^{({\bf s})}(x) &= x^8 + x^6 + 2x^5 + 2x^3 +x^2+1.
\end{align*}
The reason for this is the following. Suppose the coefficient of $x^t$ in $Q_{n-1}^{({\bf s})}(x)$ is $c(t)$. This means that
there are $c(t)~ {\bf e}$'s in $I_{n-1}^{({\bf s})}$ with $t = s_{n-1}\asc\e - e_{n-1}$. Thus, if $\asc\e= a$ and $e_{n-1} = b$, then $t = s_{n-1} a - b$. However, since $0 \leq e_{n-1} < s_{n-1}$, then $a$ and $b$ are uniquely determined by $t$. As a consequence, the coefficient of $x^{s_{n} \asc\e -\left \lfloor{\frac{s_n e_{n-1}}{s_{n-1}}} \right\rfloor }
=x^{s_n a  -\left \lfloor{\frac{s_n b}{s_{n-1}}} \right\rfloor }$ in $P_{n-1}^{({\bf s})}(x)$ is also $c(t)$.
This is because $\frac{s_n b}{s_{n-1} } =  \frac{s_n e_{n-1}}{s_{n-1} }< s_n$ implies that
\begin{align*}
t = s_n a - \left \lfloor{\frac{s_n b}{s_{n-1}}} \right\rfloor  =  s_n a' - \left \lfloor{\frac{s_n b'}{s_{n-1}}} \right\rfloor
\end{align*}
if and only if $a = a'$ and $b = b'$.

What about the missing terms in  $P_{n-1}^{({\bf s})}(x)$? These are the powers of $x$ which are in forbidden residue classes
modulo $s_n$. For instance, consider the example above with  ${\bf s} = (1,2,4,6)$.
There are no missing terms in $Q_{3}^{({\bf s})}(x) = x^5 + x^4 + 2x^3 + 2x^2 + x + 1$ since $Q_{3}^{({\bf s})}(x) $ is unimodal.
All the exponents for powers of $x$ in $P_{3}^{({\bf s})}(x) $ have the form $6a -\left \lfloor \frac{6e_3}{4} \right\rfloor$.
Thus, as $e_3$ assumes the values $\{0, 1, 2, 3\}$, then $\left \lfloor \frac{6e_3}{4} \right\rfloor$ assumes the values $\{0, 1, 3, 4\}$.
Hence, the only exponents for powers of $x$ that appear in $P_{3}^{({\bf s})}(x) $ are congruent to $\{0, -1, -3, -4 \} = \{0, 2, 3, 5\} \pmod 6$.
In other words, the coefficient $c(t)$ of $x^t$ is 0 if $t \equiv 1~ \text{or}~ 4 \pmod 6$.

In general, for ${\bf s} = (s_1, s_2, \ldots)$, the only exponents of powers of $x$ in  $P_{n-1}^{({\bf s})}(x)$ have the
form $s_n a - \left \lfloor \frac{s_n e_{n-1}}{s_{n-1}} \right\rfloor$ for $0 \leq e_{n-1} < s_{n-1}$. This leaves
$s_n - s_{n-1}$ residue classes modulo $s_n$ for the powers of $x$ that are missing in $P_{n-1}^{({\bf s})}(x)$.

Let us now see what happens when we form
\begin{equation*}
Q_n^{({\bf s})}(x) = P_{n-1}^{({\bf s})}(x) \ax.
\end{equation*}
Multiplying  $P_{n-1}^{({\bf s})}(x)$   by $\ax$ results in the coefficients of $Q_n^{({\bf s})}(x)$
being {\it interval sums} (of length $s_n$) of the coefficients of $P_{n-1}^{({\bf s})}(x)$. For instance, going back to our example with  ${\bf s} = (1, 2, 4, 6)$, we have $P_{3}^{({\bf s})}(x) = x^8 + x^6 + 2x^5 + 2x^3 +x^2+1$.
Hence, the coefficient sequence is
\[
(\ldots, 0, 0, 0, 1, {\bf 0}, 1, 2, {\bf 0} , 2, 1, {\bf 0}, 1, 0, 0, \ldots),
\]
where we have written the internal $0$'s in bold as ${\bf 0}$. If we now take the interval sums of length 6 for this sequence,
we obtain the sequence $( \ldots, 0, 0, 1, 1, 2, 4, 4, 6, 6, 6, 6, 4, 4, 2, 1, 1, 0, 0, \ldots)$. Thus,
\begin{align*}
Q_4^{(1,2,4,6)}(x)&=x^{13}+x^{12}+2x^{11}+4x^{10}+4x^9+6x^8+6x^7\\
&+6x^6+6x^5+4x^4+4x^3+2x^2+x+1.
\end{align*}

The reason that the coefficient sequence is unimodal comes from Lemma~\ref{lem:Ron_Fact1}. The change in going from an interval sum
$c(i)+c(i+1)+c(i+2)+c(i+3)+c(i+4)+c(i+5)$ to the next interval sum $c(i+1)+c(i+2)+c(i+3)+c(i+4)+c(i+5)+c(i+6)$, where we
assume that $i \geq 0$, is
just the difference $c(i+6) - c(i)$. We know that whenever $c(i) = {\bf 0}$ then this difference must be 0. Furthermore, every
interval sum of length 6 must contain exactly two {\bf 0}'s. Consequently, the sums you get are just the same sums you would get
by taking interval sums of length 4 for the coefficients of $Q_{3}^{(1,2,4)}(x) = x^5 + x^4 + 2x^3 + 2x^2 + x + 1$, but with periodic repeated coefficients because of the {\bf 0}'s. In this case,
\begin{align*}
Q_{3}^{(1,2,4)}(x)[4]_{x} = x^8+2x^7+4x^6+6x^5+6x^4+6x^3+4x^2+2x+1.
\end{align*}
In particular, Lemma~\ref{lem:Ron_Fact1} and induction imply that $Q_4^{(1,2,4,6)}(x)$ is unimodal.

The general argument follows in the same way. Namely, each interval sum of length $s_n$ for the coefficients
of  $P_{n-1}^{({\bf s})}(x)$ has exactly $s_n - s_{n-1}$ {\bf 0}'s. Furthermore, all the coefficients in these
$s_n - s_{n-1}$ residue classes modulo $s_n$ are 0. Hence, when we take interval sums of length $s_n$
for  $P_{n-1}^{({\bf s})}(x)$, we get the same coefficients as taking interval sums of length $s_{n-1}$
for the coefficients of $Q_{n-1}^{({\bf s})}(x)$, again with periodic repetitions because of the {\bf 0}'s.
By Lemma~\ref{lem:Ron_Fact1}, and by the inductive assumption that $Q_{n-1}^{({\bf s})}(x)$ is unimodal, we see that
$Q_{n-1}^{({\bf s})}(x) \ax$ is unimodal. Finally, it is easy to see
that unimodality is preserved with insertion of occasional repeated terms in the sequence.
The point is that our coefficient sequence is formed by starting at 0, adding nonnegative terms for a while
to get subsequent terms, and then at some point adding nonpositive terms to get the remaining terms of the sequence.
This shows that $Q_{n}^{({\bf s})}(x)$ is also unimodal. This completes the proof for Case 2.

\noindent{\bf Case 3.} ($s_{n-1} > s_n$): As before, we begin with an example. Let us take ${\bf s} = (5,7,3)$.
In this case, we find that
\begin{align*}
 Q_{2}^{({\bf s})}(x) &= x^{12}+2x^{11}+2x^{10}+3x^9+4x^8+4x^7+5x^6+4x^5+3x^4+3x^3+2x^2+x+1,\quad\textnormal{and}\\
P_{2}^{({\bf s})}(x) &= x^6+4x^5+7x^4+13x^3+6x^2+3x+1.
\end{align*}

How did the coefficients from $P_{2}^{({\bf s})}(x)$ arise from those in $Q_{2}^{({\bf s})}(x)$?
If we write the coefficients of $Q_{2}^{({\bf s})}(x)$ as a sequence ${\bf w}=(1,2,2,3,4,4,5,4,3,3,2,1,1)$ then the coefficients of
$P_{2}^{({\bf s})}(x)$ are formed by taking certain consecutive interval sums from ${\bf w}$ as follows:
\begin{align*}
{\bf 1} =1, \;{\bf 4} = 2+2, \;{\bf 7} = 3+4, \;{\bf 13} = 3+5+4, \;{\bf 6} = 3+3, \;{\bf 3} = 2+1, \;{\bf 1} = 1.
\end{align*}

Why is this so? As we saw in Case 2, a term {\bf e} $\in I_2^{({\bf s})}$ with $\asc\e= a$ and $e_{2} = b$ will contribute to the coefficient $c(t)$ of $x^t$ in $Q_{2}^{({\bf s})}(x)$ where $t = 7a - b$.
This {\bf e} will also
contribute to the coefficient of $x^u$ in $P_{2}^{({\bf s})}(x)$ where $u = 3 a  -\left \lfloor \frac{3b}{7} \right\rfloor$.
In fact, for $b = 0,1$ and 2, the value of $ \left \lfloor \frac{3b}{7} \right\rfloor$ is constant (and equal to 0).
Thus, the coefficient of $x^u$ in $P_{2}^{({\bf s})}(x)$ will be the sum of the three coefficients  $c(7a), c(7a-1)$
and $c(7a-2)$ in $Q_{2}^{({\bf s})}(x)$. Similarly, $ \left \lfloor \frac{3b}{7} \right\rfloor = 1$ for $b=3$ and $4$,
and $ \left \lfloor \frac{3b}{7} \right\rfloor = 2$ for $b=5$ and $6$. This partition of the interval $\{0,1,2,3,4,5,6\}$ into 3 subintervals
$\{0,1,2\}, \{3,4\}, \{5,6\}$ is repeated modulo 7. Remembering that we are subtracting $ \left \lfloor \frac{3b}{7} \right\rfloor = 1$ for $b=3$ and $4$
when computing the contributions to the coefficients, we see for example that the coefficient for $13x^3$ in
$P_{2}^{({\bf s})}(x)$ comes from the three terms $4x^7 +5x^6 +4x^5$ in $Q_{2}^{({\bf s})}(x)$,
since $3 = 3\cdot 1 - 0$ and $7 = 7 \cdot 1 -0, 6 = 7 \cdot 1 - 1$ and $5 = 7 \cdot 1 - 2$. All the other coefficients
arise in a similar manner.

The general argument follows in the same way. Namely, the interval $\{0,1,2, \ldots, s_{n-1} \}$ is partitioned
into $s_n$ consecutive subintervals, depending on the value of $ \left \lfloor \frac{s_n b}{s_{n-1}} \right\rfloor $. This results in an accumulation of consecutive coefficient values from $Q_{n-1}^{({\bf s})}(x)$ when computing
the coefficients of $P_{n-1}^{({\bf s})}(x)$. Now, to compute $Q_{n}^{({\bf s})}(x) = P_{n-1}^{({\bf s})}(x) \ax$,
we take interval sums of length $s_n$ of the coefficients of $P_{n-1}^{({\bf s})}(x)$. This is where Lemma~\ref{lem:Ron_Fact2} comes in.
It guarantees that the change in computing consecutive interval values is always nonnegative for a while and
then at some point it becomes nonpositive for the remaining terms. This proves the unimodality of $Q_{n}^{({\bf s})}(x)$ for Case~3.
\end{proof}

\section{When is $Q_n^{({\bf s})}(x)$ palindromic?}\label{sec:palindromicity}

In this section we provide a characterization of sequences $\s$ for which $Q_n^{({\bf s})}(x)$ is palindromic.

Recall that a polynomial $P(x) = \sum_{k=0}^n a_k x^k$ is palindromic if $a_k = a_{n-k}$ for $0 \leq k \leq n$, where we assume that $a_n  > 0$ and $a_0 \neq 0$. Thus, a $P(x)$ of degree $n$ is palindromic if and only if
\begin{equation*}
P(x) = x^n P(1/x).
\end{equation*}

\begin{lem}\label{lem:Ron_Fact3} Let $P(x)$ and $Q(x)$ be polynomials and let $R(x) = P(x) Q(x)$. If any two of these polynomials are palindromic then the third polynomial is also palindromic.
\end{lem}

\begin{proof} Clearly, $\deg\left(R(x)\right)=\deg\left(P(x)\right)+\deg\left(Q(x)\right)$. If $P(x)$ and $Q(x)$ are palindromic, then
\begin{align*}
R(x) &= P(x)Q(x)\\
&= x^{\deg(P(x))}P(1/x)x^{\deg(Q(x))}Q(1/x)\\
&= x^{\deg(R(x))}R(1/x),
\end{align*}
so $R(x)$ is also palindromic. Now, if $R(x)$ and $P(x)$ are palindromic, then
\begin{align*}
P(x)Q(x)=R(x)&=x^{\deg(R(x))}R(1/x)\\
&=x^{\deg(P(x))+\deg(Q(x))}P(1/x)Q(1/x)\\
&=P(x)x^{\deg(Q(x))}Q(1/x).
\end{align*}
Thus, $Q(x)=x^{\deg(Q(x))}Q(1/x)$, which means that $Q(x)$ is palindromic. By symmetry, if $R(x)$ and $Q(x)$ are palindromic, then so is $P(x)$.
\end{proof}

By Corollary~\ref{cor:PensylSavage}, we know that the degree of $Q_{n}^{({\bf s})}(x)$ will be the maximum value of $s_{n} \asc\e - e_{n}$, for $e\in\In$. To maximize the degree of $x$ in $Q_{n}^{({\bf s})}(x)$, it is always best to use the largest possible value of $\asc\e$. This means that
\begin{equation}\label{eq:e_star}
\deg\left(\Qn(x)\right)=s_{n} \asc{\bf e}^{*} - e^{*}_{n},
\end{equation}
where ${\bf e}^{*}$ denotes the element of $I_{n}^{({\bf s})}$ that achieves the largest possible value of $\asc{\bf e}$ and such that $e^{*}_{n}\leq e_{n}$ for all $\e\in\In$ satisfying $\asc\e=\asc{\bf e}^{*}$. If $1\leq j<n$, then $(e^{*}_1,e^{*}_2,\ldots,e^{*}_j)$ is the element in $\e\in I^{}_{j}$ which maximizes $s_{j} \asc\e - e_{j}$, so
\[
\deg\left(Q^{(\s)}_{j}(x)\right)=s_{j} \asc(e^{*}_1,e^{*}_2,\ldots,e^{*}_j) - e^{*}_{j}.
\]

Throughout the rest of the paper, $\e^{*}$ will be as in Equation~\eqref{eq:e_star}, unless otherwise stated.

\begin{thm} \label{thm:palindromic}
$Q_1^{({\bf s})}$ is palindromic for all ${\bf s}$.  For $n \geq 2$, write $\frac{s_n}{s_{n-1}} = \frac{p}{q}$ where $\gcd(p,q) = 1$.
Then $Q_{n}^{({\bf s})}(x)$ is palindromic if and only if $Q_{n-1}^{({\bf s})}(x)$
is palindromic  and $p e^{*}_{n-1} \equiv -1 \pmod q$.
\end{thm}

\begin{proof} The proof for $n=1$ is immediate since it is not hard to see that in this case
\begin{align*}
Q_{1}^{({\bf s})}(x) = 1+x+x^2+\ldots +x^{s_1-1},
\end{align*}
which is certainly palindromic. As in the proof of Theorem~\ref{ThmUnimodal}, we assume by induction that $Q_{n-1}^{({\bf s})}(x)$ is palindromic and we consider three cases.

\noindent{\bf Case 1.}  ($s_n = s_{n-1}$): In this case we have by Equations~\eqref{AS} and~\eqref{l3},
\begin{align*}
P_{n-1}^{({\bf s})}(x) &= \sum_{{\bf e} \in I_{n-1}^{({\bf s})}} x^{s_n \asc\e -\left \lfloor{\frac{s_n e_{n-1}}{s_{n-1}}} \right \rfloor }\\
&= \sum_{{\bf e} \in I_{n-1}^{({\bf s})}} x^{s_{n-1} \asc\e - e_{n-1}} = Q_{n-1}^{({\bf s})}(x).
\end{align*}
Hence,
\begin{align*}
Q_{n}^{({\bf s})}(x) &= P_{n-1}^{({\bf s})}(x) \ax\\
&=Q_{n-1}^{({\bf s})}(x) \ax,
\end{align*}
which is a product of palindromic polynomials and so, by Lemma~\ref{lem:Ron_Fact3}, it is also palindromic.

\noindent{\bf Case 2.} ($s_n > s_{n-1}$): The degree of $Q_{n-1}^{({\bf s})}(x)$ is $d=s_{n-1} \asc(e^{*}_1,e^{*}_2,\ldots,e^{*}_{n-1}) - e^{*}_{n-1}$.
Since $Q_{n-1}^{({\bf s})}(x)$ is palindromic, any term of the form $c(t)x^t$ has a symmetrical mate $c(t) x^{d-t}$,
where $t=s_{n-1} a - b$ and $d-t = s_{n-1} a' - b'$. These two terms will correspond to the terms
$s_n a - \left \lfloor {\frac{s_n b}{s_{n-1}}} \right \rfloor$ and $s_n a' - \left \lfloor {\frac{s_n b'}{s_{n-1}}} \right \rfloor$
in $P_{n-1}^{({\bf s})}(x)$. Since $P_{n-1}^{({\bf s})}(x)$ has degree $s_n \asc(e^{*}_1,e^{*}_2,\ldots,e^{*}_{n-1}) - \left \lfloor{\frac{s_n e^{*}_{n-1}}{s_{n-1}}} \right\rfloor$
then $P_{n-1}^{({\bf s})}(x)$ will be palindromic provided
\begin{align*}
s_n a - \left \lfloor {\frac{s_n b}{s_{n-1}}} \right \rfloor + s_n a' - \left \lfloor {\frac{s_n b'}{s_{n-1}}} \right \rfloor  =
s_n \asc(e^{*}_1,e^{*}_2,\ldots,e^{*}_{n-1}) - \left \lfloor{\frac{s_n e^{*}_{n-1}}{s_{n-1}}} \right\rfloor.
\end{align*}
Here, either $a + a' = \asc(e^{*}_1,e^{*}_2,\ldots,e^{*}_{n-1})$ and $b+b' = e^{*}_{n-1}$, or $a + a' = \asc(e^{*}_1,e^{*}_2,\ldots,e^{*}_{n-1}) -1$ and $b+b' = e^{*}_{n-1} - s_{n-1}$. In either case, in order for $P_{n-1}^{({\bf s})}(x)$ to be palindromic, we need
\begin{align} \label{palin2}
 \left \lfloor {\frac{s_n b}{s_{n-1}}} \right \rfloor +  \left \lfloor {\frac{s_n (e^{*}_{n-1} - b)}{s_{n-1}}} \right \rfloor  = \left \lfloor{\frac{s_n e^{*}_{n-1}}{s_{n-1}}} \right\rfloor
\end{align}
for all $b$. Let us reduce the fraction $\frac{s_n}{s_{n-1}} = \frac{p}{q}$ to lowest terms, so that $\gcd(p,q) = 1$. Thus, Equation~\eqref{palin2} becomes
\begin{align} \label{palin3}
 \left \lfloor {\frac{p b}{q}} \right \rfloor +  \left \lfloor {\frac{p (e^{*}_{n-1} - b)}{q}} \right \rfloor  = \left \lfloor{\frac{p e^{*}_{n-1}}{q}} \right\rfloor.
\end{align}
Setting $t = pb$ and $p e^{*}_{n-1} = mq+r$, with $0 \leq r < q$, Equation~\eqref{palin3} becomes
\begin{align} \label{palin4}
 \left \lfloor {\frac{t}{q}} \right \rfloor +  \left \lfloor {\frac{r-t}{q}} \right \rfloor  = 0
 \end{align}
for $t \in \mathbb{Z}$. However, this can hold if and only if $r \equiv p e^{*}_{n-1} \equiv -1 \pmod q$. This completes Case 2.

\noindent{\bf Case 3.} ($s_n < s_{n-1}$): In fact, this case is quite similar to Case 2. The main difference is that in generating
the coefficients of $P_{n-1}^{({\bf s})}(x)$ from those of $Q_{n-1}^{({\bf s})}(x)$, instead of having missing powers of $x$,
the coefficients of $P_{n-1}^{({\bf s})}(x)$ are (interval) sums of coefficients of $Q_{n-1}^{({\bf s})}(x)$, as in Case~3 of the proof of Theorem~\ref{ThmUnimodal}. This is because the exponent $s_{n} \asc{\bf e} - \left \lfloor {\frac{s_n e_{n-1}}{s_{n-1}}} \right \rfloor$
only changes (as $e_{n-1}$ moves) when $s_n e_{n-1}$ goes past a new multiple of $s_{n-1}$. For example, if $s_n= 3, s_{n-1}=8$ then
the exponent is the same for $e_{n-1} = 0, 1$ and $2$, provided we have the same value of $\asc{\bf e}$. Now we can use the same
argument as before to argue that the palindromicity of $Q_{n-1}^{({\bf s})}(x)$ implies the palindromicity of $P_{n-1}^{({\bf s})}(x)$ provided Equation~\eqref{palin4} holds. As in Case 2, it is necessary and sufficient that $p e^{*}_{n-1} \equiv -1 \pmod q$. This completes the proof of Case 3 and so, Theorem~\ref{thm:palindromic} is proved.
\end{proof}

In particular, Theorem~\ref{thm:palindromic} implies that $T_{n}(x)=Q_{n}^{(1,2,\ldots,n)}(x)$ is palindromic for any $n\geq 1$, which was remarked by Chung and Graham~\cite{ChungGraham2013}.

\begin{cor}[Chung, Graham~\cite{ChungGraham2013}]\label{cor:1,2,3_palindromic} Let $n\geq 1$. Then $T_{n}(x)=Q_{n}^{(1,2,\ldots,n)}(x)$ is palindromic.
\end{cor}
\begin{proof} We proceed by induction on $n$. Clearly, $T_{1}(x)=1$ is palindromic. Suppose that $T_{n-1}(x)$ is palindromic. By Theorem~\ref{thm:palindromic}, it suffices to show that $n e^{*}_{n-1} \equiv -1 \pmod {n-1}$. The inversion sequence $\e\in I_{n}^{(1,2,\ldots ,n)}$ which maximizes $n\asc\e -e_{n}$ is ${\bf e}^{*}=(0,1,\ldots, n-1)$, so $e^{*}_{n-1}=n-2$. Hence, $n e^{*}_{n-1}=n(n-2)$ and since $n^{2}-2n+1=(n-1)^{2}$, it follows that $n e^{*}_{n-1} \equiv -1 \pmod {n-1}$.
\end{proof}

\section{Another characterization of sequences $\s$ for which $\Qn (x)$ is palindromic}\label{sec:Gorenstein_cond_equiv}

Beck et al.~\cite[Thm.~4.1]{GorensteinPaper} observe that $Q_n^{({\bf s})}(x)$, defined as a generating function for the lattice points of the cone ${\bf C}^{(\s)}_{n}$, as in Equation~\eqref{eq:FirstDefQ}, is palindromic if and only if ${\bf C}^{(\s)}_{n}$ is Gorenstein, see~\cite[p.~5]{GorensteinPaper} for the definition. They also provide a characterization for ${\bf C}^{(\s)}_{n}$ being Gorenstein~\cite[Cor.~2.4]{GorensteinPaper}, which yields the following result.

\begin{thm}[Beck et al.~\cite{GorensteinPaper}]\label{thm:Beck_et_al} Let $\s$ be a sequence of positive integers and let ${\bf c}=(c_1,c_2,\ldots,c_n)$ be the sequence defined by $c_1=1$ and
\begin{equation*}
c_{i}s_{i-1} = c_{i-1}s_{i}+\gcd(s_{i-1},s_{i})
\end{equation*}
for $i>1$. Then the polynomial $Q_n^{({\bf s})}(x)$ is palindromic if and only if ${\bf c}$ consists of integer entries.
\end{thm}

Since the conditions from Theorems~\ref{thm:palindromic} and~\ref{thm:Beck_et_al} are both characterizations of the palindromicity of $\Qn(x)$, they must be equivalent to each other.

We now provide a direct proof of this fact. Suppose that we are given a sequence $(s_1,s_2,\ldots,s_{n-1})$ for which $Q_{n-1}^{(s_1,s_2,\ldots,s_{n-1})}(x)$ is palindromic. We would like to know when this sequence can be extended to $(s_1,s_2,\ldots,s_{n-1},s_{n})$ so that $Q_{n}^{(s_1,s_2,\ldots,s_{n})}(x)$ is palindromic. The condition from Theorem~\ref{thm:palindromic} is that
\begin{equation}\label{eq:our_condition}
e^{*}_{n-1}p +1\equiv 0 \pmod q,
\end{equation}
where
\[
\frac{s_{n}}{s_{n-1}}=\frac{p}{q}\quad\textnormal{and}\quad\gcd(p,q)=1.
\]
Now, the condition from Theorem~\ref{thm:Beck_et_al} depends on the sequence ${\bf c}$ consisting of integer sequences. Assuming by induction that the entries of $(c_1,c_2,\ldots,c_{n-1})$ are all integers, we see that
\[
c_{n}=\frac{1}{s_{n-1}} (c_{n-1}s_{n}+\gcd(s_{n-1},s_{n}))
\]
is an integer if and only if
\begin{equation}\label{eq:c_n_equiv}
c_{n-1}p +1\equiv 0 \pmod q.
\end{equation}
The two conditions~\eqref{eq:our_condition} and~\eqref{eq:c_n_equiv} are equivalent if and only if
\begin{equation}\label{eq:e_c_equiv}
e^{*}_{n-1} \equiv c_{n-1} \pmod q.
\end{equation}

We see that, in order to prove that the conditions from Theorems~\ref{thm:palindromic} and~\ref{thm:Beck_et_al} are equivalent, it suffices to verify that Equation~\eqref{eq:e_c_equiv} holds whenever $(c_1,c_2,\ldots,c_{n-1})$ consists of integer entries. We will prove a stronger result, namely, that
\[
e^{*}_{n-1} \equiv c_{n-1} \pmod {s_{n-1}}.
\]
To do this, we need a few technical lemmas.

\begin{lem}\label{fact:2} Suppose that $n\geq 2$ and that ${\bf c}$ consists of integer entries. Assume that $e^{*}_{n-1}\equiv c_{n-1}\pmod {s_{n-1}}$. Then the choice of $x_n$ satisfying
\[
x_{n}s_{n-1} = e^{*}_{n-1}s_n +\gcd(s_{n-1},s_n)
\]
yields the smallest integer $x_{n}$ such that $\frac{e_{n-1}}{s_{n-1}}<\frac{x_{n}}{s_{n}}$.
\end{lem}
\begin{proof}
Note that
\[
\frac{e^{*}_{n-1}}{s_{n_1}} < \frac{e^{*}_{n-1}}{s_{n-1}}+\frac{\gcd(s_{n-1}s_{n})}{s_{n-1}s_{n}} = \frac{x_{n}s_{n}}{s_{n-1}s_{n}} = \frac{x_n}{s_n}.
\]
Now, $(x_{n}-1)s_{n-1}=e^{*}_{n-1}-s_{n-1}+\gcd(s_{n-1},s_{n})$, so
\[
\frac{e^{*}_{n-1}}{s_{n-1}} \geq \frac{e^{*}_{n-1}}{s_{n-1}}+\left(\frac{\gcd(s_{n-1},s_{n})}{s_{n-1}s_{n}}-\frac{1}{s_{n}}\right) = \frac{x_{n}-1}{s_{n}},
\]
where the inequality follows from the fact that $s_{n-1}\geq \gcd(s_{n-1},s_{n})$. Hence, it suffices to show that $x_{n}$ is an integer to conclude the proof.

Since $e^{*}_{n-1}\equiv c_{n-1}\pmod {s_{n-1}}$, there exists an integer $b$ such that $s_{n-1}b = e^{*}_{n-1}-c_{n-1}$ and so, $c_{n}s_{n-1} = c_{n-1}s_{n}+\gcd(s_{n-1},s_{n})$ implies that
\begin{align*}
s_{n-1}s_{n}b+c_{n}s_{n-1} &= s_{n-1}s_{n}b+c_{n-1}s_{n}+\gcd(s_{n-1},s_{n})\\
&= e^{*}_{n-1}s_{n}+\gcd(s_{n-1},s_{n}).
\end{align*}
We conclude that $e^{*}_{n-1}s_{n}+\gcd(s_{n-1},s_{n})$ is divisible by $s_{n-1}$.
\end{proof}

Henceforth, $x_{n}$ is as in Lemma~\ref{fact:2}. Observe that, for $n\geq 2$, the integer $x_{n}$ satisfies $1\leq x_{n}\leq s_{n}$. Indeed, since $e^{*}_{n-1}<s_{n-1}$, then $x_{n}s_{n-1}-\gcd(s_{n-1},s_{n})<s_{n}s_{n-1}$. Thus, $x_{n}<s_{n}+\frac{\gcd(s_{n-1},s_{n})}{s_{n-1}}\leq s_{n}+1$.

The next lemma explains how the entries of the sequence $\e^{*}$ are related to the integers $x_{i}$.

\begin{lem}\label{lem:Carla_Observation2}  Suppose that ${\bf c}$ consists of integer entries. Then
\[
e^{*}_n = \left\{  \begin{array}{lll}
1 & \mbox{if  $n=1$ and $s_1 > 1$;} \\
0 & \mbox{if $s_n = 1$ or $x_n=s_n$;}\\
x_n &\mbox{otherwise.}
\end{array}
\right .
\]
\end{lem}

\begin{proof}
The case $n=1$ is clear. Assume $n>1$. If $x_{n} = s_n$, then Lemma~\ref{fact:2} implies that $n-1$ cannot be an ascent of $\e^{*}$. By definition, $\e^{*}$ is the choice of $\e\in\In$ that maximizes $s_{n}\asc\e-e_{n}$, so it must be that $e^{*}_{n}=0$.

Now, if $x_{n}<s_{n}$, then Lemma~\ref{fact:2} and the definition of $\e^{*}$ imply that $n-1$ is an ascent of $\e^{*}$ and that $e^{*}_{n}=x_{n}$
\end{proof}

\begin{prop}\label{prop:equiv_charact} Suppose that the sequence ${\bf c}$ consists of integer entries. Then $e^{*}_{n} \equiv c_{n} \pmod {s_{n}}$.
\end{prop}

\begin{proof} We proceed by induction on $n$. The case $n=1$ is clear because, by Lemma~\ref{lem:Carla_Observation2}, $e^{*}_{1},c_{1}=1$ if $s_{1}>1$; and $e^{*}_{1}=0,~c_{1}=1$ if $s_{1}=1$. Consider $n>1$. Suppose that if $(c_1,c_2,\ldots,c_{n-1})$ consists of integer sequences, then $e^{*}_{n-1}\equiv c_{n-1} \pmod {s_{n-1}}$.

Assume that the sequence ${\bf c}$ consists of integer entries. By our inductive hypothesis, we know that $e^{*}_{n-1}\equiv c_{n-1} \pmod {s_{n-1}}$. There are two cases.

\noindent\textbf{Case 1.} ($x_{n}<s_{n}$): It follows from Lemmas~\ref{fact:2} and~\ref{lem:Carla_Observation2} that $e^{*}_{n}s_{n-1} = e^{*}_{n-1}s_n +\gcd(s_{n-1},s_n)$. Together with the fact that $c_{n}s_{n-1} = c_{n-1}s_{n} + \gcd(s_{n-1},s_{n})$, this implies that $(e^{*}_{n} - c_{n})s_{n-1} = (e^{*}_{n-1} - c_{n-1})s_{n}$. That is,
\[
\frac{e^{*}_{n} - c_{n}}{s_n} = \frac{e^{*}_{n-1} - c_{n-1}}{s_{n-1}}.
\]
Now, $e^{*}_{n-1} \equiv c_{n-1} \pmod {s_{n-1}}$, so the fraction on the right-hand side is an integer and consequently, so is the one on the left-hand side. We conclude that $e^{*}_{n} \equiv c_{n} \pmod {s_{n}}$.

\noindent\textbf{Case 2.} ($x_{n}=s_{n}$): By Lemma~\ref{fact:2}, $s_{n}s_{n-1} = e^{*}_{n-1}s_n +\gcd(s_{n-1},s_n)$. Given that $c_{n}s_{n-1} = c_{n-1}s_{n} + \gcd(s_{n-1},s_{n})$ and $e^{*}_{n-1}\equiv c_{n-1} \pmod {s_{n-1}}$, we know that there exists an integer $b$ such that
\[
c_{n}s_{n-1} = (e^{*}_{n-1}+s_{n-1}b)s_{n} + \gcd(s_{n-1},s_{n}) = (s_{n-1} + s_{n-1}b)s_{n}.
\]
By Lemma~\ref{lem:Carla_Observation2}, $e^{*}_{n}=0$, so $c_{n}\equiv 0=e^{*}_{n} \pmod {s_{n}}$.
\end{proof}

As we mentioned before, Proposition~\ref{prop:equiv_charact} provides a direct proof of the equivalence between the conditions from Theorems~\ref{thm:palindromic} and~\ref{thm:Beck_et_al}, which are both characterizations of palindromicity of $\Qn(x)$. In addition, it shows that the sequences $\e^{*}$ and ${\bf c}$ are linked, namely, by the equivalence $e^{*}_{n} \equiv c_{n} \pmod {s_{n}}$. We now provide an inductive construction of the integers $r_{n}$ such that $c_{n}=e^{*}_{n}+r_{n}s_{n}$.

\begin{prop}
Let $\s$ be a sequence of positive integers. Define the integer sequence $(r_0, r_1, \ldots)$ by
\[
r_i = \left \{ \begin{array}{ll}
0 & \mbox{if $i=0$,}\\
1+r_{i-1} & \mbox{if $e^{*}_{i}=0$,}\\
r_{i-1} & \mbox{otherwise.}
\end{array} \right .
\]
Suppose that ${\bf c}$ consists of integer entries. Then $\e^{*}$ and ${\bf c}$ are related by
\[
c_n = e^{*}_n + r_{n}s_{n}.
\]
\end{prop}

\begin{proof} We proceed by induction on $n$. For $n=1$,  since $c_1=1$, we need to show that $1=e^{*}_1+r_1s_1$. If $s_1=1$, then, by Lemma~\ref{lem:Carla_Observation2}, $e^{*}_1=0$.  Thus, $r_1=r_0+1=1$ and so, the equality $1=e^{*}_1+r_1s_1$ holds. Otherwise, $s_1>1$, so by Lemma~\ref{lem:Carla_Observation2}, $e^{*}_1=1$ and $r_1=r_0=0$. Hence, again  $1=e^{*}_1+r_1s_1$.

Assume that $n>1$ and that $c_{n-1} = e^{*}_{n-1} + r_{n-1}s_{n-1}$. Write
\begin{eqnarray*}
c_{n}s_{n-1} & = & c_{n-1}s_n + \gcd(s_{n-1},s_{n})\\
& = & s_n(e^{*}_{n-1} + r_{n-1}s_{n-1}) + \gcd(s_{n-1},s_{n})\\
& = & x_{n}s_{n-1} -  \gcd(s_{n-1},s_{n}) + s_{n}r_{n-1}s_{n-1} + \gcd(s_{n-1},s_{n}),
\end{eqnarray*}
where the first equality is by definition of $c_n$, the second equality follows from our inductive hypothesis, and the last inequality follows by definition of $x_n$. Thus,
\[
c_n = x_n + r_{n-1}s_n.
\]
If $e^{*}_n=0$, then, since $n>1$, we deduce that $x_n=s_n$ and $r_n=r_{n-1}+1$, so
\[
c_n = s_n + r_{n-1}s_n = s_n(r_{n-1}+1) = s_nr_n= e^{*}_n+s_nr_n.
\]
Otherwise, $e^{*}_n \not = 0$, so $r_n = r_{n-1}$ and by Lemma~\ref{lem:Carla_Observation2}, $e^{*}_n=x_n$.  So,
\[
c_n = e^{*}_n + r_{n-1}s_n = e^{*}_n + r_ns_n. \qedhere
\]
\end{proof}

\section{Concluding remarks}\label{sec:concluding_remarks}

One might wonder if something stronger could be said about the coefficient sequences of $Q_{n}^{({\bf s})}(x)$.

We say that a sequence $(a_0,a_1,\ldots,a_n)$ is {\it logarithmically concave}, or {\it log-concave} for short, if $a_{i-1}a_{i+1}\leq a_{i}^{2}$ for all $0<i<n$. A polynomial $P(x)=\sum_{k=0}^{n}a_{k}x^{k}$ is {\it log-concave} if its coefficient sequence $(a_0,a_1,\ldots,a_n)$ is log-concave. Clearly, a log-concave sequence of positive terms is unimodal, so given that the polynomials $Q_{n}^{({\bf s})}(x)$ are unimodal, it is natural to ask: are the polynomials $Q_{n}^{({\bf s})}(x)$ is log-concave in general? The answer is ``No, not always'' as the example of {\bf s} $=(1,3,4)$ shows. For this
choice of {\bf s}, we find  $Q_{3}^{(1,3,4)}(x) = x^6+2x^5+2x^4+3x^3+2x^2+x+1$,
which is neither palindromic nor log-concave.

Is it possible to characterize the integer sequences $\s$ for which $Q_{n}^{({\bf s})}(x)$ is log-concave?

\bibliographystyle{amsplain}


\end{document}